\documentclass[a4paper,12pt]{article} %[a4paper,12pt]{article} %[11pt]{amsart}
\title{Arithmetic geometry of the moduli \\ stack of Weierstrass fibrations over $\mathbb{P}^1$}
\date{}
\author{Jun--Yong Park and Johannes Schmitt}

\usepackage{amsmath,amssymb,graphicx,amsthm,amsfonts,verbatim,mathtools,thmtools,fullpage,tikz-cd,url}
\usepackage[normalem]{ulem}
\usepackage[shortlabels]{enumitem}
\usepackage[all,cmtip]{xy}
\usepackage{hyperref}

\newtheorem{thm}{Theorem}[section]
\newtheorem{Mthm}[thm]{Main Theorem}
\newtheorem{lem}[thm]{Lemma}
\newtheorem{cor}[thm]{Corollary}
\newtheorem{prop}[thm]{Proposition}
\theoremstyle{definition}
\newtheorem{defn}[thm]{Definition}

\newtheorem{ques}[thm]{Question}

\newtheorem{rmk}[thm]{Remark}

\newtheorem{exmp}[thm]{Example}

\newcommand{\iso}{\cong}

\newcommand{\Mg}{\overline{\mathcal{M}}}

\newcommand{\Me}{\overline{\mathcal{M}}_{1,1}}

\def\git{/\!/ }

\newcommand{\WeierP}{\mathcal{W}^{\mathbb{P}^1}}
\newcommand{\Weier}{\mathcal{W}}
\newcommand{\WeiB}{\overline{\mathcal{W}}}

\newcommand{\UDelta}{U_\Delta}
\newcommand{\Umin}{U_{\mathrm{min}}}
\newcommand{\Usf}{U_{\mathrm{sf}}}

\newcommand{\Pcv}{\mathcal{P}(\vec{\lambda})}

\newcommand{\Pov}{\mathcal{P}(\vec{\Lambda})}

\newcommand{\Ac}{\mathcal{A}}
\newcommand{\Bc}{\mathcal{B}}

\newcommand{\Ec}{\mathcal{E}}
\newcommand{\Hc}{\mathcal{H}}

\newcommand{\Pc}{\mathcal{P}}

\newcommand{\M}{\mathcal{M}}
\newcommand{\Zc}{\mathcal{Z}}

\newcommand{\lambdavec}{{\vec{\lambda}}}
\newcommand{\Lambdavec}{{\vec{\Lambda}}}

\newcommand{\Xc}{\mathcal{X}}

\newcommand{\Yc}{\mathcal{Y}}

\newcommand{\Q}{\mathbb{Q}}
\newcommand{\N}{\mathbb{N}}

\newcommand{\Pb}{\mathbb{P}}
\newcommand{\Ab}{\mathbb{A}}
\newcommand{\A}{\mathbb{A}}

\newcommand{\Fb}{\mathbb{F}}
\newcommand{\Gb}{\mathbb{G}}
\newcommand{\Lb}{\mathbb{L}}

\newcommand{\Zb}{\mathbb{Z}}

\newcommand{\Spec}{\mathrm{Spec}}

\DeclareMathOperator{\PGL}{PGL}

\DeclareMathOperator{\PSL}{PSL}
\DeclareMathOperator{\GL}{GL}
\DeclareMathOperator{\SL}{SL}

\DeclareMathOperator{\Stab}{Stab}
\DeclareMathOperator{\Sym}{Sym}

\DeclareMathOperator{\Hom}{Hom}

\DeclareMathOperator{\Aut}{Aut}

\begin{document}

    \maketitle

    \begin{abstract}
    Coarse moduli spaces of Weierstrass fibrations over the (unparameterized) projective line were constructed by the classical work of [Miranda] using Geometric Invariant Theory. In our paper, we extend this treatment by using results of [Romagny] regarding group actions on stacks to give an explicit construction of the moduli stack $\mathcal{W}_{n}$ of Weierstrass fibrations over an unparameterized $\mathbb{P}^{1}$ with discriminant degree $12n$ and a section. We show that it is a smooth algebraic stack and prove that for $n\geq 2$, the open substack $\mathcal{W}_{\mathrm{min},n}$ of minimal Weierstrass fibrations is a separated Deligne--Mumford stack over any base field $K$ with $\mathrm{char}(K) \neq 2,3$ and $\nmid n$. Arithmetically, for the moduli stack $\mathcal{W}_{\mathrm{sf},n}$ of stable Weierstrass fibrations, we determine its motive in the Grothendieck ring of stacks to be $\{\mathcal{W}_{\mathrm{sf},n}\} = \mathbb{L}^{10n - 2}$ in the case that $n$ is odd, which results in its weighted point count to be $\#_q(\mathcal{W}_{\mathrm{sf},n}) =  q^{10n - 2}$ over $\mathbb{F}_q$. In the appendix, we show how our methods can be applied similarly to the classical work of [Silverman] on coarse moduli spaces of self-maps of the projective line, allowing us to construct the natural moduli stack and to compute its motive.
    %to acquire its motive and point count.
    \end{abstract}

    %\vspace{-1ex}
    %\smallskip

    \tableofcontents

    \newpage

    \section{Introduction}
    \label{sec:intro}
 
    Recall that a \emph{Weierstrass fibration} over a smooth algebraic curve $C$ is a flat, proper morphism $f : X \to C$ with irreducible fibres of arithmetic genus $1$ such that the general fibre is smooth, together with a section of $f$ not passing through the singular points of the fibres. The \emph{discriminant} $\Delta$ of such a fibration is a section of a suitable line bundle on $C$ having zeros precisely at the points $t$ of $C$ with singular fibres $X_t$.
    
    By generalizing the notion above to families of Weierstrass fibrations parameterized by a scheme $S$, one can then define the moduli stack $\Weier_n$ of Weierstrass fibrations over an \textit{unparameterized} $\Pb^1$ with discriminant degree $12n$. This means we consider Weierstrass fibrations such that the curve $C$ above is isomorphic to $\Pb^1$, without including such an isomorphism as part of the data. In our paper, we give an explicit construction of the moduli stack $\Weier_n$, directly extending the seminal work of \cite[Theorem 7.2]{Miranda} where the corresponding coarse moduli space was constructed via Geometric Invariant Theory (see also \cite{Seiler}). 

    To describe the construction, recall that a Weierstrass fibration over (a parameterized) $\Pb^1$ of discriminant degree $12n$ can always be cut inside a $\Pb^2$-bundle over $\Pb^1$ by an equation
    \[
    Y^2Z = X^3 + AXZ^2+BZ^3\,,
    \]
    over $\Spec(\Zb[1/ 6])$ where $X,Y,Z$ are coordinates on the bundle and 
    \begin{equation} \label{eqn:UvecspacePov}
        (A,B) \in H^0(\Pb^1, \mathcal{O}_{\Pb^1}(4n)) \oplus H^0(\Pb^1,\mathcal{O}_{\Pb^1}(6n)) \eqqcolon V
    \end{equation}
    define the so-called \emph{Weierstrass data} of the fibration.
    The sections $A,B$ satisfy that the discriminant $\Delta = 4 A^3 + 27 B^2$ is not identically zero, and the pair $(A,B)$ is unique up to the scaling action $\lambda \cdot (A,B) = (\lambda^4 A, \lambda^6 B)$ of $\lambda \in \Gb_m$.
    
    Correspondingly, let $\mathbb{G}_m$ act on \eqref{eqn:UvecspacePov} with weight $4$ on the first summand and with weight $6$ on the second summand and consider the weighted projective stack
    \[
    \mathcal{P}(V) = [V \setminus \{0\} / \mathbb{G}_m]\,.
    \]
    Inside $\mathcal{P}(V)$ we have several open substacks that are natural given the interpretation of the points $[A,B] \in \mathcal{P}(V)$ as Weierstrass data:
    %Then we have $\Pov = [U \setminus \{0\} / \mathbb{G}_m]$. Natural homogeneous coordinates on $\Pov$ are given by $[A:B]$ with $A \in H^0(\Pb^1, \mathcal{O}_{\Pb^1}(4n))$, $B \in H^0(\Pb^1, \mathcal{O}_{\Pb^1}(6n))$ such that $(A,B) \neq (0,0)$. Using these coordinates, we define several natural open substacks of $\Pov$:
    \begin{itemize}
        \item $\UDelta \subset \mathcal{P}(V)$ is the locus where the {discriminant} $\Delta = 4 A^3 + 27 B^2$ is nonzero\footnote{~This is well-defined since $\Delta$ is homogeneous of weight $12n$ with respect to the $\mathbb{G}_m$-action.},
        \item $\Umin \subset \mathcal{P}(V)$ is the locus such that $\Delta$ is nonzero and moreover there exists no geometric point $p \in \Pb^1$ such that $\mathrm{ord}_p(A) \geq 4$ and $\mathrm{ord}_p(B) \geq 6$, 
        \item $\Usf \subset \mathcal{P}(V)$ is the locus such that there exists no geometric point $p \in \Pb^1$ such that $A(p)=B(p)=0$. % ; this is by definition the same as $\Hom_n(\Pb^1,\Me \iso \Pc(4,6))$. 
    \end{itemize}
    It is easy to see that we have strict inclusions:
    \[
    \Usf \subsetneq \Umin \subsetneq \UDelta \subsetneq \mathcal{P}(V)\,.
    \]
    Let $\WeierP_n$ be the stack parameterizing families of Weierstrass fibrations over the fixed base $\Pb^1$ of discriminant degree $12n$. There exists a natural map
    \begin{equation} \label{eqn:Phidefintro}
        \Phi : \WeierP_n \to \UDelta \subseteq \mathcal{P}(V)
    \end{equation}
    associating the Weierstrass data $[A,B]$ to a given fibration (see Section \ref{Sect:Fibparametrized} for the precise definitions of $\WeierP_n$ and $\Phi$). 

    \medskip

    In our first theorem, we show that the known correspondence between Weierstrass fibrations and Weierstrass data implies that $\Phi$ is an open embedding.
    \begin{thm}\label{Thm:Weierstrass_Moduli}
    For any degree $n \in \mathbb{Z}_{\geq 1}$ and any base field $K$ with $\mathrm{char}(K) \neq 2,3$, the moduli stack $\WeierP_n$ is isomorphic to the open substack $\UDelta \subseteq \mathcal{P}(V)$ of the weighted projective stack $\mathcal{P}(V)$ via the map $\Phi$ from \eqref{eqn:Phidefintro}.
    Moreover, under this correspondence the open substacks of minimal (resp. stable) Weierstrass fibrations in $\WeierP_n$ correspond to the substack $\Umin$ (resp. $\Usf$) inside $\UDelta$.\footnote{~Again, see Section \ref{Sect:Fibparametrized} for a reminder of the definitions of minimal and stable Weierstrass fibrations.}
    %and the locus of stable Weierstrass fibrations corresponds to the substack $\Usf \subseteq \UDelta$. %\jocomment{Not 100 percent sure about this last point, but not implausible.}
    \end{thm}
    For studying the stack $\Weier_n$ of Weierstrass fibrations over an unparameterized $\Pb^1$, it is natural to consider the action of $\mathrm{Aut}(\Pb^1) = \PGL_2$ on $\WeierP_n$ by composing the Weierstrass fibration with an automorphism of $\Pb^1$. 
    It is easy to see that this action is induced by an action on the ambient weighted projective stack $\mathcal{P}(V)$. Indeed, the action of an element of $\PGL_2$ on the homogeneous coordinates $X,Y$ of $\Pb^1$ translates to an action on the global sections $A,B$ of $\mathcal{O}_{\Pb^1}(4n), \mathcal{O}_{\Pb^1}(6n)$ which are the homogeneous coordinates of $\mathcal{P}(V)$. Note that since both $\WeierP_n$ and $\mathcal{P}(V)$ are themselves stacks, the formal definition of these actions requires one to use the notion of group actions on stacks presented in \cite{Romagny}.
    
    The next result then confirms that the moduli stack $\Weier_n$, whose objects are Weierstrass fibrations over an (unparameterized) family of smooth rational curves, is indeed given as the quotient under the action above. This concrete construction allows us to conclude many nice properties of this stack (and the substacks of minimal and stable fibrations).

    %\jocomment{Think again about $\Umin \subseteq \UDelta$ issue is 2 vs 3 compared to 4 vs 6?}

    %We first recall that the proper Deligne--Mumford stack of stable elliptic curves $\Me \iso \Pc(4,6)$ over $\Spec(\Zb[1/ 6])$. To describe the moduli stack $\Weier_n$ and its various open substacks we choose particular coordinates on $\Pov$. For a given degree $n \in \Nb$ consider the following vector space:
    %Our main result is that ----- over $K$. Precisely:

    \begin{Mthm}\label{Thm:Weierstrass_Stack_Quotient}
    Fix a degree $n \in \Zb_{\geq 1}$ and a base field $K$ with $\mathrm{char}(K) \neq 2,3$. 
    \begin{enumerate}[label=\alph*)]
        \item The moduli stack $\Weier_n$ (for $n\geq 1$) of Weierstrass fibrations over an unparameterized $\mathbb{P}^{1}$ with discriminant degree $12n$ is isomorphic to the quotient
        \[
        \Weier_n \cong [\WeierP_n / \PGL_2] = [U_{\Delta} / \PGL_2]\,,
        \]
        which is naturally an open substack of the stack $[\mathcal{P}(V) / \PGL_2]$. In particular, $\Weier_n$ is a smooth, irreducible algebraic stack of finite type with affine diagonal.
        \item Inside $\Weier_n$, the open substacks  $\Weier_{\mathrm{min},n}$ (for $n\geq 2$) of minimal Weierstrass fibrations and $\Weier_{\mathrm{sf},n}$ (for $n\geq 1$) of stable Weierstrass fibrations are separated Deligne--Mumford stacks for $\mathrm{char}(K) \nmid n$, which are tame for $\mathrm{char}(K)=0$ or $>12n$.
    \end{enumerate}
    \end{Mthm}
    First, let us note that some bounds on the characteristic are indeed necessary to guarantee tameness of the stabilizer groups in the second part of the theorem. Indeed, the Weierstrass fibration associated to $[A:B] = [X^{4n}: Y^{6n}] \in \mathcal{P}(V)$ is invariant under scaling $X$ by an element of $\mu_{4n}$ and by scaling $Y$ under an element of $\mu_{6n}$. Together, these transformations generate a copy of $\mu_{12n}$ inside $\PGL_2$ which acts as an automorphism of the fibration, so that the automorphism group is not tame when $\mathrm{char}(K)$ divides $12n$. For the missing case $\Weier_{\mathrm{min},n}$ in part b) above, see Remark \ref{Rmk:Weiermin1}.
    
    Second, we can make a comparison to the treatment in \cite{Miranda}. There, Miranda considers the weighted projective space $\mathbb{P}(V)$, the coarse moduli space of $\mathcal{P}(V)$, and takes a GIT-quotient under the action of $\SL_2$ induced by the $\PGL_2$-action above. Away from the strictly semistable points, the resulting variety is the coarse moduli space of a suitable open substack of $\Weier_n$. 

    \medskip
    
    % One crucial observation is that in the setting of \cite{Miranda}, the $\Gb_m$-quotient in the construction of $\mathbb{P}(V)$ and the subsequent $\SL_2$-quotient can be performed in a single step by considering the natural $\GL_2$-action on the space $V$. This is fine as far as coarse moduli spaces are concerned, but it does \emph{not} result in the same stack. The reason is that the diagonal subtorus $\Gb_m \subseteq \GL_2$ acts with weights $4n$ and $6n$ on the summands of $V$ in \eqref{eqn:UvecspacePov}, whereas - as Theorem \ref{Thm:Weierstrass_Moduli} shows - the natural weights from the point of view of Weierstrass fibrations are $4,6$. 
    % Emphasize here that the above result vindicates our usage of the vector $(4,4,\ldots, 6,6,\ldots)$ instead of $(4n,4n, \ldots, 6n, 6n, \ldots)$.
    
    One of the advantages of working with the moduli stacks of Weierstrass fibrations instead of their coarse moduli spaces is that they naturally allow us to define some meaningful arithmetic invariants. One of them is their \emph{weighted point count} over the finite field $\Fb_q$, where each $\Fb_q$-point $x$ of the stack is weighted by $1/|\mathrm{Aut}(x)|$. 
    When the degree $n$ is odd, we can not only compute this count for  the moduli stack $\Weier_{\mathrm{sf},n}$, but in fact we can obtain an even finer invariant of this stack, namely its motivic class inside $K_0(\mathrm{Stck}_{K})$ the \emph{Grothendieck ring of $K$-stacks} introduced by \cite{Ekedahl}.
    \begin{Mthm} \label{Thm:WeierstrassMotive}
    Let $\mathrm{char}(K) \neq 2,3$ and fix an odd degree $n$. Then the motive $\{\Weier_{\mathrm{sf},n}\} \in K_0(\mathrm{Stck}_{K})$ is given by
        \[\{\Weier_{\mathrm{sf},n}\} = \Lb^{10n - 2}\;,\]
    where $\Lb = \{\mathbb{A}^1\}$ is the \emph{Lefschetz motive}. Thus we acquire the weighted point count over $\Fb_q$ as
        \[\#_q\left(\Weier_{\mathrm{sf},n}\right) = q^{10n - 2}\;.\]
    Moreover, from the inclusions
    \[
    \Weier_{\mathrm{sf},n} \subsetneq \Weier_{\mathrm{min},n} \subsetneq \Weier_n \subsetneq \left[\mathcal{P}(V) / \PGL_2\right]\,
    \]
    we then obtain the following bounds
    \begin{align*}
    q^{10n - 2} = \#_q\left(\Weier_{\mathrm{sf},n}\right)  \le \#_q\left(\Weier_{\mathrm{min},n}\right) & \le \#_q\left(\Weier_{n}\right) \leq \#_q\left( \left[\mathcal{P}(V) / \PGL_2\right]   \right) = \frac{q^{10n+2}-1}{q(q-1)(q^2-1)} \\
    & = q^{10n-2} + q^{10n-3} + 2 q^{10n-4} + 2 q^{10n-5} + O(q^{10n-6}) \,.    
    \end{align*}
    % which implies (for $n$ odd) that the weighted point counts over $\Fb_q$ with $\mathrm{char}(\Fb_q) > 12n$ of the intermediate open substacks $\Weier_{\mathrm{min},n}$ and $\Weier_n$ are naturally bounded as follows:
    % \[
    % q^{10n - 2} \le \#_q\left(\Weier_{\mathrm{min},n}\right) \le \#_q\left(\Weier_{n}\right) \leq q^{10n-2} + q^{10n-3} + 2 q^{10n-4} + 2 q^{10n-5} + O(q^{10n-6}) \,.
    % \]
    \end{Mthm}
    %their motives in $K_0(\mathrm{Stck}_{K})$ and 
    %    \[
    %\Lb^{10n - 2} \le \{\Weier_{\mathrm{min},n}\} \le \{\Weier_{n}\} < \Lb^{10n-2} + \Oc(\Lb^{10n-3}) \,, 
    %\]
    First, let us note that in the proof of the above result, the restriction to odd degree $n$ comes from the fact that this allows us to convert the $\PGL_2$-quotient $\left[\mathcal{P}(V) / \PGL_2\right]$ into a (twisted) $\GL_2$-quotient $[V / \GL_2]$. This is crucial for the computation in the Grothendieck ring since, in contrast to $\PGL_2$, the group $\GL_2$ is special and thus behaves much nicer in computations of motives (see Section \ref{sec:Motive} for details). However, it is natural to expect that the result above holds verbatim for even degrees $n$.

    \medskip
    
    %\jocomment{I had a bunch of comments and questions about the following two paragraphs (since I think they could be formulated more clearly). Here is my attempt at rephrasing them.}
    
    Second, in the above theorem it is quite interesting that the shape of the motive of the moduli stack $\Weier_{\mathrm{sf},n}$ is so simple, and just given by the power of Lefschetz motive $\mathbb{L}^{\dim \Weier_{\mathrm{sf},n}}$. Given that the coarse moduli space of $\Weier_{\mathrm{sf},n}$ is rational by \cite{Lejarraga}, it is a valid question whether this equality of motives can be obtained from a suitable locally-closed filtration of the moduli stack by affine pieces.

    \medskip
    
    Finally, one could consider the natural problem of computing arithmetic invariants of $\Weier_{\mathrm{min},n}$. Here, inside the missing piece $\Weier_{\mathrm{min},n} \setminus \Weier_{\mathrm{sf},n}$ we start seeing the appearance of additive singular fibers, and we might expect to obtain a stratification geometrically determined by \textit{Tate's algorithm} in relation to the \textit{Kodaira--N\'eron classification} of singular fibers. And in this regard, it would be interesting to see whether the mixed Tate property of the moduli stacks persists for the larger open substack $\Weier_{\mathrm{min},n}$. If so, this would imply the emergence of lower order terms $q^m$ (for $m \leq 10n-3$) in the weighted point count $\#_q\left(\Weier_{\mathrm{min},n}\right)$ over $\Fb_q$.
    We will study the arithmetic geometry of the moduli stack of \textit{unstable elliptic surfaces} in a forthcoming work.

    \subsection*{Methods \& Techniques}
    It turns out that many of the above results follow more generally from studying moduli stacks of maps from a parameterized (or unparameterized) $\Pb^1$ to some algebraic stack $\M$. The connection to Weierstrass fibrations relies on the isomorphism 
    $$(\Me)_K \cong [ (\mathrm{Spec}~K[a,b]-(0,0)) / \Gb_m ] = \Pc_K(4,6)$$
    between the moduli stack $\Me$ of stable curves of genus $1$ with a marked point with the weighted projective stack $\Pc(4,6)$, which works over $\Spec(\Zb[1/ 6])$ (see \cite[Proposition 3.6]{Hassett}). Then, given a morphism $f : \Pb^1 \to \Me$, we obtain a Weierstrass fibration over $\Pb^1$ by pulling back the universal curve $\overline{\mathcal{C}}_{1,1} \to \Me$ (and its section) along $f$. The discriminant degree $12n$ is then uniquely determined by the property that $f^* \mathcal{O}_{\Pcv}(1) = \mathcal{O}_{\Pb^1}(n)$. It is not hard to see that \emph{any} Weierstrass fibration can be uniquely determined like this - essentially the Weierstrass data $[A:B]$ is \emph{precisely} the data needed to define the morphism to $\Pc_K(4,6)$ (see Section \ref{Sect:Fibparametrized}). Thus the moduli stack of stable Weierstrass fibrations over (a parameterized) $\Pb^1$ is isomorphic to the Hom stack (c.f. \cite{Olsson})
    $$\Hom_n(\Pb^1,\Pcv) \coloneqq \left\{f\colon  \Pb^1 \to \Pcv : f \text{ morphism with }f^* \mathcal{O}_{\Pcv}(1) = \mathcal{O}_{\Pb^1}(n) \right\}$$
    for $\lambdavec = (4,6)$. 
    Here in general, we write $\Pcv=[(\Ab_{x_0, \dotsc, x_N}^{N+1}\setminus 0)/\Gb_m]$ for the weighted projective stack where $\Gb_m$ acts with weights 
    $$\lambdavec = (\lambda_0, \ldots, \lambda_N) \in \mathbb{Z}_{\geq 1}^{N+1},\ (N \geq 1)\,.$$
    In the following, we consider these $\Hom$-stacks for arbitrary targets $\Pcv$, which have been previously studied in \cite{HP, PS, HP2}. Many of the weighted projective stacks $\Pcv$ have modular interpretations (see  Example~\ref{exmp:many_mod}), resulting in a modular interpretation of $\Hom_n(\Pb^1,\Pcv)$ by pulling back a suitable universal family as before. See the end of the section for a discussion in the case $\Pc(1,3) \iso \Me[\Gamma_1(3)]$ of the moduli stack of generalized elliptic curves with $[\Gamma_1(3)]$-level structure.
    
    %The idea of \textit{studying rational curves on a fine moduli stack} is interesting as an algebraic stack $\M$ carries the universal family $\Uc \rightarrow \M$ that allows one to pullback a family of objects through a given morphism $\varphi_f: \Pb^1 \rightarrow \M$ inducing a unique fibration over $\Pb^1$ (i.e., flat proper morphism $f : X \rightarrow \Pb^1$). While there are many different moduli stacks one could consider in this regard; the canonically convex smooth projective algebraic stack is the \textit{weighted projective stack} $\Pcv\coloneqq[(\Ab_{x_0, \dotsc, x_N}^{N+1}\setminus 0)/\Gb_m]$ where for the weight $\vec{\lambda} = (\lambda_0, \dotsc, \lambda_N)$ with positive weights $\lambda_i \in \N$, $\zeta \in \Gb_m$ acts by $\zeta \cdot (x_0, \dotsc, x_N)=(\zeta^{\lambda_0} x_0, \dotsc, \zeta^{\lambda_N} x_N)$. Geometrically, $\Pcv$ is appealing as it is minimal (i.e., Picard rank 1) and for various modular weights $\vec{\lambda}$ many important classes of Deligne--Mumford stacks are isomorphic to $\Pcv$ as in Example~\ref{exmp:many_mod} over a field $K$ with $\text{char}(K)$ does not divide $\lambda_i \in \N$ for every $i$.

    %If we assemble the algebraic stack of morphisms then we are led to consider the Hom stack $\Hom_n(\Pb^1,\Pcv) \coloneqq \left\{f : \Pb^1 \to \Pcv : f \text{ morphism with }f^* \mathcal{O}_{\Pcv}(1) = \mathcal{O}_{\Pb^1}(n) \right\}$ of degree $n \in \N$ rational curves from the parameterized $\Pb^1$ to $\Pcv$ formulated previously in \cite{HP, HP2, PS}.
    Many of the phenomena that we saw for the spaces of Weierstrass fibrations (related to $\mathcal{P}(4,6)$) now generalize to arbitrary $\Pcv$. First, a morphism to $\Pcv$ is determined by a nonzero element of
    \[
    (u_0, \ldots, u_N) \in \bigoplus_{i=0}^N H^0(\Pb^1, \mathcal{O}_{\Pb^1}(\lambda_i n)) = V\,,
    \]
    up to the scaling of $\Gb_m$ of weight $\lambda_i$ on $u_i$. Thus $\Hom_n(\Pb^1,\Pcv)$ is naturally an open substack of the weighted projective stack $\mathcal{P}(V)$.\footnote{~More generally, given any Deligne--Mumford stack $\Xc$ which can be embedded into a weighted projective stack $\Xc \hookrightarrow \Pcv$, the map $\Hom_n(\Pb^1,\Xc) \hookrightarrow \Hom_n(\Pb^1,\Pcv) \subsetneq \Pov$ represents the Hom stack $\Hom_n(\Pb^1,\Xc)$ as a locally closed substack of $\Pov$ the associated ambient weighted projective stack. The existence of a similar class of ambient algebraic stacks for arbitrary non-cyclotomic Deligne--Mumford stacks $\Xc$ is much less clear.} We denote this stack by $\Pov = \mathcal{P}(V)$ to stress the dependence on $\lambdavec$. Then the natural action of $\PGL_2 = \mathrm{Aut}(\Pb^1)$ on $\Hom_n(\Pb^1,\Pcv)$ (by precomposition) extends to an action on $\Pov$ and using \cite{Romagny} we can define quotients
    %Geometrically, the Hom stack $\Hom_n(\Pb^1,\Pcv)$ has an excellent property that it naturally embeds as an open substack of the associated ambient weighted projective stack $\Pov$.\footnote{~More generally, given any Deligne--Mumford stack $\Xc$ which can be embedded into a weighted projective stack $\Xc \hookrightarrow \Pcv$, the map $\Hom_n(\Pb^1,\Xc) \hookrightarrow \Hom_n(\Pb^1,\Pcv) \subsetneq \Pov$ represents the Hom stack $\Hom_n(\Pb^1,\Xc)$ as a locally closed substack of $\Pov$ the associated ambient weighted projective stack. The existence of a similar class of ambient algebraic stacks for arbitrary non-cyclotomic Deligne--Mumford stacks $\Xc$ is much less clear.} In this paper, we will consider the $\Aut(\Pb^1) = \PGL_2$ stack quotient on the Hom stack to formulate the moduli stack of degree $n \in \N$ rational curves from an \textit{unparameterized} $\Pb^1$ to $\Pcv$ which naturally sets the framework for constructing the fine moduli stack of Weierstrass fibrations over $\Pb^1$.
    \begin{equation*}
    \M_{n}^{\lambdavec} \coloneqq \left[\Hom_n(\Pb^1,\Pcv)~/~\PGL_2\right] \subseteq \left[\Pov~/~\PGL_2\right]\,.
    \end{equation*}
    Below, we use GIT-techniques from \cite{GIT}, applied to the action of $\PGL_2$ on the coarse moduli space $\mathbb{P}(\Lambdavec)$ of $\Pov$, to both obtain a moduli space for $\M_{n}^{\lambdavec}$ and to construct some natural (partial) compactifications. For this purpose let
    \[
    \mathbb{P}(\Lambdavec)^{s} \subseteq \mathbb{P}(\Lambdavec)^{ss} \subseteq \mathbb{P}(\Lambdavec)\ \text{ and }\ \Pov^{s} \subseteq \Pov^{ss} \subseteq \Pov
    \]
    be the loci of (semi)stable points for the natural linearization of the $\PGL_2$-action and denote by
    \[
    M_{n}^{\lambdavec,s} \subseteq M_{n}^{\lambdavec,ss}\ \text{ and }\ \M_{n}^{\lambdavec,s} \subseteq \M_{n}^{\lambdavec,ss} \subseteq [\Pov/\PGL_2]
    \]
    the associated GIT (resp. stack) quotients under $\PGL_2$. In Proposition \ref{Prop:stablesemistablecrit} we give an explicit criterion for (semi)stability, in particular showing that $\Hom_n(\Pb^1,\Pcv)$ is always contained in the stable locus and allowing us to write $M_n^{\lambdavec} \subseteq M_{n}^{\lambdavec,s}$ for its image in $M_{n}^{\lambdavec,s}$.
    Then combining techniques from GIT with an analysis of the $\PGL_2$-stabilizers on $\Pov$, we obtain Theorem \ref{Thm:Moduli_GIT} below. Here, the characteristic of the base field will again be important, and the behaviour is particularly nice when
    \begin{equation} \label{eqn:charKintro}
        \mathrm{char}(K) = 0 \ \text{ or } \mathrm{char}(K) > \max_{0 \leq i<j\leq N}\mathrm{lcm}(\lambda_i, \lambda_j) \cdot n\,.
    \end{equation}
    The geometric significance of the maximum over $\mathrm{lcm}(\lambda_i, \lambda_j) \cdot n$ is that this number is the greatest degree that a map $\Pb^1 \to \Pcv$ can have onto its image.
    \begin{thm}\label{Thm:Moduli_GIT}
    Fix a base field $K$ and a vector $\lambdavec \in \mathbb{Z}_{\geq 1}^{N+1}$ for $N \geq 1$.
    \begin{enumerate}[label=\alph*)]
        \item The stack $[\Pov/\PGL_2]$ and its open substacks $\M_n^{\lambdavec}$, $\M_{n}^{\lambdavec,s}$, $\M_{n}^{\lambdavec,ss}$ are smooth, irreducible algebraic stacks of finite type with affine diagonals.
        \item The stacks $\M_n^{\lambdavec}$ and $\M_{n}^{\lambdavec,s}$ are separated and for  $\mathrm{char}(K)$ satisfying \eqref{eqn:charKintro}, they are tame Deligne--Mumford stacks.\footnote{~In case of $\M_n^{\lambdavec}$, the Deligne--Mumford property in fact only requires that $\mathrm{char}(K)$ does not divide any of the $\lambda_i$ or $n$.}
        \item The schemes $M_{n}^{\lambdavec}, M_{n}^{\lambdavec,s}$ and $M_{n}^{\lambdavec,ss}$ are all integral, normal varieties, with $M_{n}^{\lambdavec,ss}$ projective and $M_{n}^{\lambdavec}, M_{n}^{\lambdavec,s}$ are the coarse moduli spaces for $\M_{n}^{\lambdavec}, \M_{n}^{\lambdavec,s}$.
        \item In the case that both degree $n$ and all weights $\lambda_i$ are odd, we have that stable and semistable points coincide and that $\M_{n}^{\lambdavec, s}=\M_{n}^{\lambdavec, ss}$ is a proper algebraic stack.
    \end{enumerate}
        % There exists a commutative diagram of morphisms of Deligne--Mumford moduli stacks %and their coarse moduli spaces
        % \begin{equation}
        % \begin{tikzcd}
        % \Hom_n(\Pb^1,\Pcv) \arrow[d, "\varphi"] \arrow[r, hookrightarrow]& \Pov^s \arrow[d, "\varphi^s"] \arrow[r, hookrightarrow]& \Pov^{ss}\arrow[d, "\varphi^{ss}"] \arrow[r, hookrightarrow] & \Pov \arrow[d, "/\PGL_2"] \\
        % \M_{n}^{\lambdavec} \arrow[d, "c"] \arrow[r, hookrightarrow]& \M_{n}^{\lambdavec, s} \arrow[d, "c^s"] \arrow[r, hookrightarrow]& \M_{n}^{\lambdavec,ss} \arrow[d, "c^{ss}"] \arrow[r, hookrightarrow] &  \left[\Pov / \PGL_2\right]  \\
        % M_{n}^{\lambdavec}\arrow[r, hookrightarrow] & M_{n}^{\lambdavec,s} \arrow[r, hookrightarrow] & M_{n}^{\lambdavec,ss}
        % \end{tikzcd}
        % \end{equation}
        % where the horizontal arrows are open embeddings and where the entries of the bottom row are varieties. The maps $c \circ \varphi, c^s \circ \varphi^s$ are geometric quotients and the map $c^{ss}\circ \varphi^{ss}$ is a categorical quotient for the action of $\PGL_2$. 
    \end{thm} 

    For the original case of $\lambdavec = (4,6)$, the presence of strictly semistable points inside $\Pov$ prevents us from obtaining a proper moduli stack compactifying the moduli stack of stable Weierstrass fibrations. On the other hand, we note that the moduli stack $\Me[\Gamma_1(3)]$ of generalized elliptic curves with $[\Gamma_1(3)]$-level structure has an isomorphism $\Me[\Gamma_1(3)] \iso \Pc(1,3)$ as in \cite[Proposition 4.5]{HMe} through the equation 
    \[
    Y^2Z + AXYZ + BYZ^2 = X^3\,,
    \]
    over $\Spec(\Zb[1/ 3])$. As all the modular weights of $\Pc(1,3)$ are indeed odd numbers, we naturally acquire Theorem \ref{Thm:Level_SP} as the particularly interesting application of Theorem \ref{Thm:Moduli_GIT} that for an odd degree $n$ the stack
    \[
    \WeiB_{n}^{(1,3)} \coloneqq \M_{n}^{(1,3), s} = \M_{n}^{(1,3), ss}
    \]
    over a base field $K$ with $\mathrm{char}(K) \neq 3$ and $\nmid n$ is a smooth, proper and irreducible Deligne--Mumford stack of dimension $4n-2$ which admits a projective coarse moduli space. It is also tame for $\mathrm{char}(K)=0$ or $>3n$. This suggests that the proof of Main Theorem \ref{Thm:Weierstrass_Stack_Quotient}, which uses the analogous isomorphism $\Me \iso \Pc(4,6)$, could naturally be generalized to show that the stack $\M_n^{(1,3)}$ has a modular interpretation as the moduli stack of stable Weierstrass fibrations with $[\Gamma_1(3)]$-level structure and discriminant degree $12n$ over an unparameterized $\mathbb{P}^{1}$, and that $\WeiB_{n}^{(1,3)}$ renders a \textit{smooth modular compactification} of this moduli stack via GIT. While we do not pursue this direction further for now, we remark that the explicit construction of $\WeiB_{n}^{(1,3)}$ as a quotient stack would in particular open an avenue to study its intersection theory in future work.

    \medskip

    Lastly, our computations of motives and point counts presented in Main Theorem \ref{Thm:WeierstrassMotive} can be generalized for $\M_n^{\lambdavec}$ (see Theorem \ref{Thm:motivequotient1}).
    There we see that while the formula is more complicated than in Main Theorem \ref{Thm:WeierstrassMotive}, the motive $\{\M_n^{\lambdavec}\}$ is still of mixed Tate type. Concerning the weighted point counts over $\Fb_q$, the proof again implies that
    \[
        \left\{ \left[\Pov / \PGL_2\right] \right\} = \frac{\Lb^{|\vec{\lambda}|n+N+1}-1}{\Lb(\Lb-1)(\Lb^2-1)}\,,
    \]
    so that for any open substack $\M_n^{\lambdavec} \subseteq \mathcal{U} \subseteq [\Pov / \PGL_2]$ we have
    \[
    \#_q\left(\M_{n}^{\lambdavec}\right) \leq \#_q\left(\mathcal{U}\right) \leq \frac{q^{|\vec{\lambda}|n+N+1}-1}{q(q-1)(q^2-1)}\,.
    \]

    As a final remark, let us mention that the methods of our paper can be adapted to the study of moduli spaces of self-maps of $\Pb^1$ from \cite{Silverman}. For these spaces, one considers the action of $\PGL_2$ on $\Hom_n(\Pb^1, \Pb^1)$ given by \emph{conjugation}
    \[
    \varphi \cdot f = \varphi \circ f \circ \varphi^{-1} \text{ for }\varphi \in \PGL_2, f \in \Hom_n(\Pb^1, \Pb^1)\,.
    \]
    %, where an element $\varphi \in \PGL_2$ sends $f \in \Hom_n(\Pb^1, \Pb^1)$ to $\varphi \circ f \circ \varphi^{-1}$. 
    Then the paper \cite{Silverman} studies the GIT-quotient of this action. In Appendix \ref{sec:Appen_Sil} we explain how the methods above can be used to show that the corresponding quotient stack
    \[
    \M_n = \left[\Hom_n(\Pb^1, \Pb^1) / \PGL_2 \right]
    \]
    is a smooth, separated and tame Deligne--Mumford stack for $\mathrm{char}(K)=0$ or $\mathrm{char}(K)>n$ whose motive (for \emph{even} degree $n$) is given by $\{\M_n\} = \mathbb{L}^{2n-2}$ (see Theorem \ref{Thm:motivequotientSilverman}).

    \subsection*{Outline of the paper}
    In Section \ref{sec:Rational_Curves} we discuss the stacks of maps from the projective line to a weighted projective stack. We treat the case of the parameterized domain $\Pb^1$ in Section \ref{sec:ModuliofHom} and construct the space with an unparameterized domain in Section \ref{sec:PGL_Moduli} using a $\PGL_2$-stack quotient. We finish the proof of Theorem \ref{Thm:Moduli_GIT} in Section \ref{sec:GIT} using techniques from classical GIT. In Section \ref{sec:Motive} we recall the theory of Grothendieck rings of stacks and weighted point counts over $\Fb_q$ and prove Theorem \ref{Thm:WeierstrassMotive}.
    In Section \ref{Sect:Weierstrass} we present an abstract definition of the stacks of Weierstrass fibrations over the projective line and prove that these stacks are given by the stacks of morphisms to $\mathcal{P}(4,6)$ constructed in Section \ref{sec:Rational_Curves}. In particular, this allows us to conclude the proofs of Main Theorems \ref{Thm:Weierstrass_Stack_Quotient} and \ref{Thm:WeierstrassMotive}.
    
    In Appendix \ref{sec:Appen_Sil} we explain the applications of our methods to the spaces of self-maps of the projective line. Finally, in Appendix \ref{Sect:stabfatpoints} we give a self-contained discussion of the $\PGL_2$-stabilizer groups of subschemes of $\Pb^1$ given by unions of fat points, which is needed to show the Deligne--Mumford property of some of the quotient stacks we construct.

    \section{Moduli of maps from rational curves to weighted projective stacks}
    \label{sec:Rational_Curves} %on \texorpdfstring{$\Pcv$}{P(lambda)} 

    \subsection{Hom-stacks of parameterized rational curves}
    \label{sec:ModuliofHom} %\texorpdfstring{$\Hom_n(\Pb^1,\Pcv)$}{Hom n(P1,P(lambda))}  on \texorpdfstring{$\Pcv$}{P(lambda)}

    \par In this section, we recall basic facts regarding the $N$-dimensional weighted projective stack $\Pcv$ with the weight vector $\vec{\lambda} = (\lambda_0, \dotsc, \lambda_N)$ and the Hom-stack $\Hom_n(\Pb^1,\Pcv)$ of morphisms $\Pb^1 \to \Pcv$ from a parameterized $\Pb^1$ over a basefield $K$ with $\mathrm{char}(K) \nmid \lambda_i \in \N$ for every $i$.

    \begin{defn}\label{def:wtproj} 
    Let $\lambdavec = (\lambda_0, \ldots, \lambda_N) \in \mathbb{Z}_{\geq 1}^{N+1}$ be a vector of $N+1$ positive integers. Consider the affine space $U_\lambdavec = \Ab_{x_0, \dotsc, x_N}^{N+1}$ endowed with the action of $\Gb_m$ with weights $\lambdavec$ where an element $\zeta \in \Gb_m$ acts by
    \begin{equation}
    \zeta \cdot (x_0, \ldots, x_N) = (\zeta^{\lambda_0} x_0, \ldots, \zeta^{\lambda_N} x_N)\,.
    \end{equation}
    The $N$-dimensional weighted projective stack $\Pcv$ is then defined as the quotient stack
    \[
    \Pcv = \left[(U_{\lambdavec} \setminus \{0\})/\Gb_m \right]\,.
    \]
    \end{defn}

    \par It is a smooth Deligne--Mumford stack. For $\lambdavec = (1, \ldots, 1)$ we recover the classical projective space $\Pcv = \mathbb{P}^N$. The natural morphism $U_\lambdavec \to \Pcv$ is the total space of the \emph{tautological line bundle} $\mathcal{O}_{\Pcv}(-1)$ on $\Pcv$. As in the classical case, we denote by $\mathcal{O}_{\Pcv}(1)$ the dual of this line bundle.

     \par Note that $\Pcv$ is not an (effective) orbifold when $\gcd(\lambda_0, \dotsc, \lambda_N) \neq 1$. In this case, the finite cyclic group scheme $\mu_{\gcd(\lambda_0, \dotsc, \lambda_N)}$ is the generic stabilizer of $\Pcv$. When we need to emphasize the field $K$ of definition of $\Pcv$, we instead use the notation $\Pc_{K}(\vec\lambda)$. Note that when $K=\Fb_p$ for some prime $p$, the stack $\Pc(1,p)$ is not Deligne--Mumford, as $\underline{\Aut}_{[0:1]} \cong \mu_p$, which is not formally unramified over $\Fb_p$. Nevertheless, the following proposition shows that any $\Pcv$ behaves well in most characteristics as a tame Deligne--Mumford stack:
    
    \begin{prop}\label{prop:wtprojtame}
        The weighted projective stack $\Pcv = \Pc(\lambda_0, \dotsc, \lambda_N)$ is a tame Deligne--Mumford stack over $K$ if $\mathrm{char}(K)$ does not divide $\lambda_i \in \N$ for every $i$.
    \end{prop}
    
    \begin{proof}
        For any algebraically closed field extension $\overline{K}$ of $K$, any point $y \in \Pcv(\overline{K})$ is represented by the coordinates $(y_0,\dotsc,y_N) \in \Ab_{\overline{K}}^{N+1}$ with its stabilizer group as the subgroup of $\Gb_m$ fixing $(y_0,\dotsc,y_N)$. Hence, any stabilizer group of such $\overline{K}$-points is $\mu_u$ where $u$ divides $\lambda_i$ for some $i$. Since the characteristic of $K$ does not divide the orders of $\mu_u$ for any $i$, the stabilizer group of $y$ is $\overline{K}$-linearly reductive. Hence, $\Pcv$ is tame by \cite[Theorem 3.2]{AOV}. Note that the stabilizer groups constitute fibers of the diagonal $\Delta: \Pcv \rightarrow \Pcv \times_K \Pcv$. Since $\Pcv$ is of finite type and the groups $\mu_u$ are unramified over $K$ whenever $u$ is not divisible by $\mathrm{char}(K)$, the map $\Delta$ is unramified as well. Therefore, $\Pcv$ is also Deligne--Mumford by \cite[Theorem 8.3.3]{Olsson2}.
    \end{proof}
    
    \par The tameness is analogous to flatness for stacks in positive/mixed characteristic as it is preserved under base change by \cite[Corollary 3.4]{AOV}. Moreover, if a stack $\mathcal X$ is tame and Deligne--Mumford, then the formation of the coarse moduli space $c: \mathcal X \rightarrow X$ commutes with base change as well by \cite[Corollary 3.3]{AOV}.

    \begin{exmp}\label{exmp:many_mod}
        There is a whole array of moduli stacks of curves that are isomorphic to $\Pcv$ for various weights $\vec{\lambda} = (\lambda_0, \dotsc, \lambda_N)$ over a field $K$ with $\text{char}(K)$ does not divide $\lambda_i \in \N$ for every $i$:
        
        \begin{itemize}
        \item When $\text{char}(K)\ne 2,3$, \cite[Proposition 3.6]{Hassett} shows that one example is given by the proper Deligne--Mumford stack of stable elliptic curves $$(\Me)_K \cong [ (\mathrm{Spec}~K[a_4,a_6]-(0,0)) / \Gb_m ] = \Pc_K(4,6)$$ by using the short Weierstrass equation $y^2 = x^3 + a_4x + a_6x$, where $\zeta \cdot a_i=\zeta^i a_i$ for $\zeta \in \Gb_m$ and $i=4,6$. This is no longer true when $\text{char}(K)\in \{2,3\}$, as the Weierstrass equations are more complicated. 

        \item Similarly, one could consider the stack $\Me[\Gamma]$ of generalized elliptic curves with $[\Gamma]$-level structure by the work of Deligne and Rapoport \cite{DR} (summarized in \cite[\S 2]{Conrad} and also in \cite[\S 2]{Niles}) such as $\Me[\Gamma_1(2)] \iso \Pc(2,4)$ \cite[\S 1.3]{Behrens}, $\Me[\Gamma_1(3)] \iso \Pc(1,3)$ \cite[Proposition 4.5]{HMe} and $\Me[\Gamma(2)] \iso \Pc(2,2)$ \cite[Proposition 7.1]{Stojanoska}. 
        %, $\Me[\Gamma_1(4)] \iso \Pc(1,2)$ \cite[Examples 2.1]{Meier} 

        \item Also, one could consider the stack $\Mg_{1,m}(m-1)$ of $m$-marked $(m-1)$-stable curves of arithmetic genus one formulated originally by the works of \cite{Smyth, Smyth2} such as $\Mg_{1,2}(1) \iso \Pc(2,3,4)$, $\Mg_{1,3}(2) \iso \Pc(1,2,2,3)$, $\Mg_{1,4}(3) \iso \Pc(1,1,1,2,2)$ and $\Mg_{1,5}(4) \iso \Pb(1,1,1,1,1,1) \iso \Pb^5$ as shown by \cite[Theorem 1.5.7.]{LP}.

        \item For higher genus $g \ge 2$, the moduli stack $\Hc_{2g}[2g-1]$ of quasi-admissible (i.e., monic odd degree hyperelliptic) genus $g$ curves (originally introduced by \cite{Fedorchuk}) with a generalized Weierstrass equation $y^2 = x^{2g+1} + a_{4}x^{2g-1} + a_{6}x^{2g-2} + a_{8}x^{2g-3} + \cdots + a_{4g+2}$ is the proper Deligne--Mumford stack isomorphic to $\Pc(4,6,8,\dotsc,4g+2)$ by \cite[Proposition 4.2(1)]{Fedorchuk} over $\mathrm{char}(K)=0$ and by \cite[Proposition 5.9]{HP2} over $\mathrm{char}(K)> 2g+1$.
        %The $g=2$ case with $\Hc_{4}[3] \iso \Pc(4,6,8,10)$ is of special interest as all genus 2 curves are hyperelliptic as well as all principally polarized Abelian surfaces come from Jacobians of compact genus 2 curves (i.e., genus 2 curves with dual graphs equal to tree) as shown in \cite[4. Theorem]{OU} (see also \cite[Satz 2]{Weil}).
        \end{itemize}
    \end{exmp}
    Now for $n \in \N$, consider the stack 
    \begin{equation} \label{eqn:Homnfirstdef}
    \Hom_n(\Pb^1,\Pcv) = \left\{f : \Pb^1 \to \Pcv : f \text{ morphism with }f^* \mathcal{O}_{\Pcv}(1) = \mathcal{O}_{\Pb^1}(n) \right\}
    \end{equation}
    of morphisms from $\Pb^1$ to $\Pcv$ of degree\footnote{~We remark that this notion of degree is different than the one used in papers like \cite{Vistoli}. The reason is that e.g. for $N=1$, the cycle $c_1(\mathcal{O}_\Pcv(1)) \in H_2(\Pcv)$ does not necessarily have degree $1$, but instead degree $1/D$ for $D = \mathrm{lcm}(\lambda_0, \lambda_1)$. Then for $f \in \Hom_n(\Pb^1,\Pcv)$ we have $f_*[\Pb^1] = n D \cdot [\Pcv] \in H_2(\Pcv)$ in the sense of \cite{Vistoli}. However, in this paper we will exclusively use the notion of degree used in \eqref{eqn:Homnfirstdef}.} $n$. Such a morphism $f$ can be specified by a tuple
    \begin{equation} \label{eqn:uspace}
    \vec{u} = (u_0, \ldots, u_N) \in \bigoplus_{i=0}^N H^0\left(\Pb^1, \mathcal{O}_{\Pb^1}(n \lambda_i)\right)
    \end{equation}
    of sections of the bundles $\mathcal{O}_{\Pb^1}(n \lambda_i)$ such that the $u_i$ do not have a common zero on $\Pb^1$. The tuple $\vec{u}$ is unique up to the scaling by elements $\zeta \in \Gb_m$ defined by
    \begin{equation} \label{eqn:uaction}
    \zeta \cdot (u_0, \ldots, u_N) = (\zeta^{\lambda_0} u_0, \ldots, \zeta^{\lambda_N} u_N)\,.
    \end{equation}
    From this description, we see that we can identify $\Hom_n(\Pb^1,\Pcv)$ as an open substack of an ambient larger weighted projective stack. Consider the vector
    \[
    \Lambdavec = (\underbrace{\lambda_0, \ldots, \lambda_0}_{n\lambda_0+1 \text{ times}}, \ldots, \underbrace{\lambda_N, \ldots, \lambda_N}_{n\lambda_N+1 \text{ times}} )\in \mathbb{Z}^{\sum_i (n\lambda_i+1)} \,.
    \]
    As the dimension of the space of sections of $\mathcal{O}_{\Pb^1}(n\lambda_i)$ is $n\lambda_i+1$, it is easy to see that the space \eqref{eqn:uspace} of tuples $\vec{u}$ with its $\Gb_m$-action \eqref{eqn:uaction} is naturally identified with $U_\Lambdavec$. Thus we can see $\Hom_n(\Pb^1,\Pcv)$ as the open substack
    \begin{equation} \label{eqn:HomsubsetPLambda}
    \Hom_n(\Pb^1,\Pcv) = \left\{[\vec u] : \text{the } u_i \text{ have no common zero on }\Pb^1 \right\} \subseteq \Pov\,.
    \end{equation}
    From this description, it is clear that $\Hom_n(\Pb^1,\Pcv)$ itself has the structure of a smooth Deligne--Mumford stack.

    \medskip

    \subsection{Stacks of morphisms from unparameterized rational curves}\label{sec:PGL_Moduli} %\texorpdfstring{$\M_{n}^{\lambdavec}$}{Mnlambda\coloneqq[Hom n(P1, P(lambda))/PGL2]}   on \texorpdfstring{$\Pcv$}{P(lambda)}

    \par To study the moduli stack $\M_{n}^{\lambdavec}$ of morphisms $\Pb^1 \to \Pcv$ from an \emph{unparameterized} $\Pb^1$, we want to divide the Hom stack $\Hom_n(\Pb^1,\Pcv)$ by the automorphism group $\PGL_2$ acting on the domain of the morphism. On the level of sets, this action is easy to describe: an automorphism $\varphi \in \PGL_2$ of $\Pb^1$ acts on $f \in \Hom_n(\Pb^1,\Pcv)$ by sending it to $f \circ \varphi^{-1}$.\footnote{~As usual, the inverse is necessary to obtain a left action.} 

    Note that since $\Hom_n(\Pb^1,\Pcv)$ itself already has the structure of a Deligne--Mumford stack, we formally need to use the notion of algebraic group actions on stacks developed in \cite{Romagny} to lift this action from the level of sets to the algebraic category. It is straightforward to see that all the necessary compatibility relations from \cite{Romagny} are satisfied (see \cite[Lemma C.6]{Schmitt} for a related check). 

    We can also give an explicit description of the above action using coordinates. For this, recall that the action of $\PGL_2$ on $\Pb^1$ is given by
    \begin{equation} \label{eqn:matrixactioncoordinates}
    \begin{pmatrix}
    a & b \\ c & d
    \end{pmatrix} \cdot [X, Y] = [aX + bY, cX + dY]\,.
    \end{equation}
    Given $[\vec u] \in \Hom_n(\Pb^1,\Pcv)$, the entry $u_i \in H^0(\Pb^1,\mathcal{O}_{\Pb^1}(n \lambda_i))$ can be identified with a homogeneous polynomial
    \begin{equation} \label{eqn:coordinaterepr}
    u_i = a_{i,0} X^{n \lambda_i} + a_{i,1} X^{n \lambda_i-1} Y + \ldots + a_{i,n \lambda_i} Y^{n \lambda_i}
    \end{equation}
    in the coordinates $X,Y$. 
    Then from the action \eqref{eqn:matrixactioncoordinates} we can explicitly see the action of elements of $\PGL_2$ on the coordinates $a_{i,j}$ of $\Hom_n(\Pb^1,\Pcv)$. In fact, the coordinates $a_{i,j}$ can be seen as the homogeneous coordinates on the large weighted projective stack $\Pov$ containing $\Hom_n(\Pb^1,\Pcv)$, and the action of $\PGL_2$ naturally extends to $\Pov$.
    We define
    \begin{equation} \label{eqn:Mlambda}
        \M_{n}^{\lambdavec} \coloneqq \left[\Hom_n(\Pb^1,\Pcv)~/~\PGL_2\right]\,.
    \end{equation}
    Then we have the following theorem.
    \begin{thm}\label{Thm:QuotisDM}
    Let $\mathcal{V} \subseteq \Pov$ be an open substack on which $\PGL_2$ acts with finite, reduced stabilizers at geometric points. Then the quotient stack $[\mathcal{V}/\PGL_2]$ is a smooth Deligne--Mumford stack of finite type with affine diagonal over $K$.
    If $\mathrm{char}(K)$ divides neither one of the weights $\lambda_i$ nor any of the orders of the stabilizer groups of the $\PGL_2$ action, then the quotient stack is also tame.
    In particular, for $\mathrm{char}(K) = 0$ or greater than \eqref{eqn:maxdegree}, 
    this condition is satisfied for $\mathcal{V} = \Hom_n(\Pb^1,\Pcv)$.
    \end{thm}    
    Our main goal in this section is to prove Theorem \ref{Thm:QuotisDM}. For this, we need to analyze the stabilizer groups of the $\PGL_2$-action on $\Pov$. We begin by observing the following technical result.
    \begin{lem} \label{Lem:stabilizerextension}
    Let $G$ be a flat, separated group scheme of finite presentation and let $\mathcal{M}$ be an algebraic stack with a $G$-action in the sense of \cite{Romagny}. Then for a geometric point $p : \mathrm{Spec}(k) \to \mathcal{M}$, the automorphism group $(\mathcal{M}/G)_{[p]}$ fits into an exact sequence
    \begin{equation}
        0 \to \mathcal{M}_{p} \to (\mathcal{M}/G)_{[p]} \to G_p \to 0
    \end{equation}
    of group schemes. Here $\mathcal{M}_{p}$ is the automorphism group of $p$ in $\mathcal{M}$ and $G_p \subseteq G$ is the \emph{stabilizer group of $p$ in $G$}, i.e., the closed subgroup of $g \in G$ such that $g p \cong p \in \mathcal{M}$.
    \end{lem}
    \begin{proof}
    This follows from \cite[Theorem 4.1]{Romagny} as explained in \cite[Remark 4.2]{Romagny}.
    \end{proof}
    
    For a geometric point $[\vec u] \in \Pov$, a \emph{base point} of $[\vec u]$ is a point $q = [X_0:Y_0] \in \Pb^1$ such that $u_i(q)=0$ for $i=0, \ldots, N$. 
    %If moreover we have $\mathrm{ord}_q(u_i) \geq \lambda_i$ for all $i$, then we can remove the base point (dividing $u_i$ by $(X Y_0 - Y X_0)^{\lambda_i}$). After removing
    For $q_1, \ldots, q_\ell \in \Pb^1$ the finitely many base points of $[\vec u]$, denote by
    \[
    \varphi_{[\vec u]} : \Pb^1 \setminus \{q_1, \ldots, q_\ell\} \to \Pcv
    \]
    the underlying rational map to $\Pcv$.
    
    % \jocomment{Ok, in the proof below we need to be more careful, I think, and exclude some more characteristics from the start. The problem: there can be non-reduced subschemes $Z$ of $\Pb^1$, whose support includes at least three points and which are still invariant under some non-reduced subscheme of $PGL_2$. Example: $Z=V(X(X-Y)^pY)$, which is invariant under sending $[X:Y] \mapsto [X:\lambda Y]$ for $\lambda \in \mu_p$ in characteristic $p$.
    % Also, if $p$ divides the degree of a map $\Pb^1 \to X$, it can be that each point in the image has a non-reduced preimage.
    % I think that the entire problem boils down to: given (over alg. closed field) a zero-dimensional subscheme of $\Pb^1$ (given as vanishing set of some homogeneous polynomial), when can we ensure that the stabilizer (in $\PGL_2$) of this subset is finite and reduced.
    % }

    \begin{lem} \label{Lem:constmap}
     Let $[\vec u] \in \Pov$ be a geometric point. Then the underlying rational map $\varphi_{[\vec u]}$ is \emph{not} finite if and only if there exists $[\vec a] \in \Pcv$ and
     \[
     U \in H^0(\Pb^1, \mathcal{O}_{\Pb^1}(\ell \cdot n))\text{ with }\ell = \gcd(\lambda_i : i=0,\ldots, N \text{ s.t. }a_i \neq 0)\,,
     \]
     such that
     \[
     u_i = a_i \cdot U^{\lambda_i / \ell}\,.
     \]
    \end{lem}
    \begin{proof}
    Since we consider a geometric point of $\Pov$, we can assume for the entire proof that we work over an algebraically closed base field $\bar K$.
    If the $u_i$ are of the form $u_i = a_i \cdot U^{\lambda_i / \ell}$ as above, the rational map $\varphi_{[\vec u]}$ contracts $\Pb^1$ to the point $[\vec a] \in \Pcv$, so it is indeed not finite.
    Conversely, assume that $\varphi_{[\vec u]}$ is \emph{not} finite, so it contracts $\Pb^1$ to some point $[\vec b] \in \Pcv$. Let $I \subseteq \{0, \ldots, N\}$ be the subset of indices with $b_i \neq 0$. For $j \notin I$ we must have $u_j = 0$, since otherwise there is a point of $\Pb^1$ where $u_j$ does not vanish and which thus maps to a point different from $[\vec b]$. 
    
    Let $q \in \Pb^1$ be any geometric point and consider the vector $v_q = (\mathrm{ord}_q(u_i) : i \in I)$. We claim that this vector is an integer multiple
    \[
    v_q = m_q \cdot (\lambda_i/\ell : i \in I)
    \]
    of the vector of numbers $\lambda_i/\ell$. This is clear for $v_q=0$, so assume otherwise. Then a priori the rational map $\Pb^1 \to \Pcv$ could have point of indeterminacy at $q$. To resolve it, let $L = \mathrm{lcm}(\lambda_i : i \in I)$ and consider the precomposition $\widetilde \varphi$ of $\varphi_{[\vec u]}$ with the coordinate change $t \mapsto t^L$, for a local coordinate $t$ around the point $q$. Then the $i$-th homogeneous coordinate of $\widetilde \varphi$ is given by a nonzero multiple of $t^{L \cdot v_{q,i}}$. Let $M \geq 0$ be the largest integer such that all $t^{L \cdot v_{q,i}}$ are divisible by $t^{M \cdot \lambda_i}$. Then by the definition of $L$, there exists $i_0 \in I$ with $M \cdot \lambda_{i_0} = L \cdot v_{q,i_0}$. But if there was some $i' \in I$ such that $M \cdot \lambda_{i'} <  L \cdot v_{q,i'}$, then the limit $\widetilde \varphi(t)$ as $t \to 0$ would have vanishing coordinate $i'$. This would imply that this limit is different from $[\vec b]$, giving a contradiction to the map $\varphi_{[\vec u]}$ being constant. Thus indeed we have $M \cdot \lambda_{i} = L \cdot v_{q,i}$ for all $i \in I$, so that
    \[
    v_q = \frac{M}{L} \cdot (\lambda_i : i \in I)
    \]
    This shows that the vector $v_q$ is a \emph{rational} multiple of the integer vector $(\lambda_i/\ell : i \in I)$, and since the $\gcd$ of the entries of this integer vector is $1$ by construction, it follows that all factors $m_q$ are in fact integers.
    
    To finish the proof, first observe that the numbers $m_q$ we found above must satisfy that for each $i \in I$ we have
    \[
    \sum_{q \in \Pb^1} m_q \cdot \frac{\lambda_i}{\ell} = \lambda_i \cdot n \quad \iff \quad  \sum_{q \in \Pb^1} m_q = \ell \cdot n\,,
    \]
    since the vanishing divisor of the section $u_i$ must have total degree $\lambda_i \cdot n$. Now let $U \in H^0(\Pb^1, \mathcal{O}_{\Pb^1}(\ell \cdot n))$ be a section with multiplicity $m_q$ at $q$ for all $q \in \Pb^1$, which is unique up to scaling. Then the sections $u_i$ and $U^{\lambda_i/\ell}$ of $\mathcal{O}_{\Pb^1}(\lambda_i \cdot n)$ have the same divisors of zeros and thus differ by some scalar factor 
    \[u_i = a_i \cdot U^{\lambda_i/\ell}\,.\] 
    This is precisely the statement claimed in the lemma, so the proof is finished.
    \end{proof}
    
    \begin{prop} \label{Prop:finredstabilizercrit}
    Let $[\vec u] \in \Pov$ be a geometric point such that the underlying rational map $\varphi_{[\vec u]}$ is finite\footnote{~The map being finite essentially means that it is not constant, though due to stack issues we have to be careful about about maps whose image is a single point but which do not factor through $\Spec(K)$.}. Then if $\mathrm{char}(K)$ does not divide any of the $\lambda_i$ or the number $n$, the stabilizer of $[\vec u]$ in $\PGL_2$ is finite and reduced.
    
    Moreover, if either $\varphi_{[\vec u]}$ is finite or $[\vec u]$ has at least three distinct base points and $\mathrm{char}(K)=0$ or greater than
    \begin{equation} \label{eqn:maxdegree}
    \max_{0 \leq i<j\leq N}\mathrm{lcm}(\lambda_i, \lambda_j) \cdot n\,,
    \end{equation}
    then the stabilizer of $[\vec u]$ in $\PGL_2$ is finite and reduced and tame.
    \end{prop}
    \begin{proof}
    Consider the stabilizer group $(\PGL_2)_{[\vec u]}$ of the point $[\vec u]$. Then both the set $\{q_1, \ldots, q_\ell\} \subset \Pb^1$ of base points and the map $\varphi_{[\vec u]}$ defined on their complement must be invariant under this stabilizer. First assume that the map $\varphi_{[\vec u]}$ is a finite map. 
    Then by Lemma \ref{Lem:constmap} there must exist $0 \leq i < j \leq N$ such that $u_i, u_j$ do not vanish identically and so we get a well-defined rational map 
    $$
    \begin{tikzcd}
    \overline{\varphi} : \Pb^1 \arrow[r, dashed, "\varphi_{[\vec u]}"] & \Pcv \arrow[r, dashed] &\mathcal{P}(\lambda_i, \lambda_j) \arrow[r,"c"] & \Pb^1
    \end{tikzcd}
    $$
    by composing $\varphi_{[\vec u]}$ first with the projection to the $(i,j)$-coordinates and then with the coarse moduli space map $c$ of $\mathcal{P}(\lambda_i, \lambda_j)$. Then clearly the stabilizer group of $\varphi_{[\vec u]}$ must also fix $\overline{\varphi}$.
    
    Now the degree of $\overline{\varphi}$ onto $\Pb^1$ is given by $\mathrm{lcm}(\lambda_i, \lambda_j) \cdot n$. Therefore, by the assumption on the characteristic, the degree of $\overline{\varphi}$ is coprime to the characteristic of the base field, and thus $\overline{\varphi}$ is generically \'etale. Now choose $P_1, P_2, P_3$ general points in its image, then the preimages $\overline{\varphi}^{-1}(P_j)$ are disjoint unions of reduced points.
    % Argument: The map from $\Pb^1$ to its target $C$ induces a finite field extension, whose degree is coprime to $\mathrm{char}(K)$. Thus the field extension is separable and thus the map $\Pb^1 \to C$ is \'etale at the generic point. 
    % and therefore has a primitive element. It is thus given as $K(C)[x]/(f(x)) / K(C)$$, with $f$ irreducible and coprime to its derivative $f'$.   
    % \jocomment{I need the statement here that for a cover $C \to D$ of degree $d$ it is true that the cover is not ramified everywhere if the base characteristic does not divide the degree of the cover. Probably a keyword is something like "tamely ramified", but I did not find a straightforward reference.}
    
    The stabilizer acts by permutation on these preimages, so we have a group homomorphism
    \begin{equation} \label{eqn:stabsubgroupSym}
    (\PGL_2)_{[\vec u]} \to \bigoplus_{j=1}^3 \mathrm{Sym}(\varphi_{[\vec u]}^{-1}(P_j))\,,
    \end{equation}
    where the group scheme on the right is a constant (and \'etale) group scheme over the base field. We claim that the above morphism is injective. Indeed, the kernel is a subgroup of $\PGL_2$ fixing the fibres over three points. In particular it fixes three distinct points in $\Pb^1$ and so, as in the proof of Lemma \ref{Lem:reducedinclusion}, we see that the kernel is trivial (essentially using that the action of $\PGL_2$ on $\Pb^1$ is simply $3$-transitive). We conclude that $(\PGL_2)_{[\vec u]}$ is \'etale as a subgroup scheme of an \'etale group scheme. Finally note that the cardinality of the preimages $\overline{\varphi}^{-1}(P_j)$, which is the degree of $\overline{\varphi}$, is bounded by \eqref{eqn:maxdegree}. Then, in the second part of the result above, by assumption the base characteristic does not divide the order of the group on the right hand side of \eqref{eqn:stabsubgroupSym}, and thus this group is tame as well. The same then holds for the subgroup $(\PGL_2)_{[\vec u]}$.
    
    On the other hand, assume that $[\vec u]$ has at least three distinct base points. We then observe that the stabilizer group $(\PGL_2)_{[\vec u]}$ is a subgroup of the stabilizer $\Stab(Z)$ of the base locus
    $$Z = V(u_i~;~i=0, \ldots, N) \subseteq \Pb^1\,.$$
    Hence it suffices to show that $\Stab(Z)$ is finite, reduced and tame. For this, note that the bound on the base characteristic ensures that it does not divide the multiplicity of any of the base points. Thus by Proposition \ref{Pro:stabilizerfatpoints} we conclude that the stabilizer is finite and reduced. Moreover, by Lemma \ref{Lem:reducedinclusion} we have that it is in fact a subgroup of the symmetric group acting on the base points. 
    Since there are at most
    \[
    \max_{i=0, \ldots, N} \lambda_i \cdot n
    \]
    such base points, the bound on the characteristic ensures that it does not divide the order of the group. Hence we conclude that $\Stab(Z)$ and its subgroup $(\PGL_2)_{[\vec u]}$ are tame, finishing the proof.
    % (i.e., the common vanishing locus of all the sections $u_i$). 
    %  (i.e., 
    % acts by permutation on the underlying finite reduced subscheme $\{q_1, \ldots, q_\ell\} \subset \Pb^1$ and we have a group morphism
    % \[
    % (\PGL_2)_{[\vec u]} \to \bigoplus_{j=1}^3 \mathrm{Sym}(\{q_1, \ldots, q_\ell\})\,,
    % \]
    % which again is injective since $\ell \geq 3$ by the simple $3$-transitivity of the $\PGL_2$ action on $\Pb^1$. Thus the argument finishes as before. \jocomment{Again estimates on maximal size of $\ell$ for tameness.}
    \end{proof}

    \begin{cor} \label{Cor:finredtamecrit}
    Let $\mathrm{char}(K)=0$ or greater than \eqref{eqn:maxdegree} and let $[\vec u] \in \Pov$ be a geometric point. Then the $\PGL_2$-stabilizer of $[\vec u]$ is \emph{not} finite, reduced and tame if and only if the $\PGL_2$-orbit of $[\vec u]$ contains a point $[\vec v]$ such that
    \[v_i = a_i \cdot (X^e Y^f)^{\lambda_i/\ell}\,,\]
    where $[\vec a] \in \Pcv$, $\ell = \gcd(\lambda_i : i=0,\ldots, N \text{ s.t. }a_i \neq 0)$ and $e,f \geq 0$ such that $e + f = \ell \cdot n$.
    \end{cor}
    \begin{proof}
    If the stabilizer of $[\vec u]$ is {not} finite, reduced and tame, by Proposition \ref{Prop:finredstabilizercrit} we know that $[\vec u]$ has at most two base points  and the map $\varphi_{[\vec u]}$ is finite. By the $\PGL_2$ action we find a point $[\vec v]$ in the orbit of $[\vec u]$ such that its base points are contained in $\{0, \infty\}$. Then we can conclude that $[\vec v]$ has the form above using Lemma \ref{Lem:constmap}, together with the fact that the base points of $[\vec v]$ being inside $\{0, \infty\}$ implies that the section $U$ from Lemma \ref{Lem:constmap} must have the form $U = X^e Y^f$.
    Conversely, it is immediate that for $[\vec v]$ of the form above, the $\PGL_2$-stabilizer is of positive dimension (we can scale the coordinate $Y$ by some $\ell$-th power of a scalar), proving the converse direction.
    \end{proof}
    We are now ready to prove Theorem~\ref{Thm:QuotisDM}.
    \begin{proof}[Proof of Theorem~\ref{Thm:QuotisDM}]
    The stack $[\mathcal{V}/\PGL_2]$ is algebraic by \cite[Theorem 4.1]{Romagny}. Since $\mathcal{V}$ is smooth and $\mathcal{V} \to [\mathcal{V}/\PGL_2]$ is a principal $\PGL_2$-bundle (and hence also a smooth morphism), the smoothness of the quotient follows. It remains to show that it is a Deligne--Mumford stack. In the case that $\mathcal{V}$ has generically trivial stabilizer (i.e., when the $\lambda_i$ share no common factor) this follows from \cite[Proposition C.3]{Schmitt}. If this is not satisfied, we can more generally argue as follows: since we already know that $[\mathcal{V}/\PGL_2]$ is algebraic, it suffices to show that the diagonal morphism $\Delta_{[\mathcal{V}/\PGL_2]}$ is unramified, see \cite[Tag 06N3]{Stacks}. This can be checked on geometric fibers, and the non-empty geometric fibers are precisely the automorphism groups of points of $[\mathcal{V}/\PGL_2]$. 
    %\href{https://stacks.math.columbia.edu/tag/06N3}
    By Lemma \ref{Lem:stabilizerextension}, these groups are extensions of the automorphism groups of points in $\mathcal{V}$ and their stabilizer groups in $\PGL_2$. The automorphism groups of the points themselves are finite and unramified by Proposition \ref{prop:wtprojtame}, whereas the stabilizers are by assumption. Since the category of finite, unramified group schemes is closed under extensions (see \cite[Proposition 40]{Stix}), we have that the automorphism groups of geometric points of $[\mathcal{V}/\PGL_2]$ are finite and unramified, so $[\mathcal{V}/\PGL_2]$ is Deligne--Mumford.
    
    The additional condition that $\mathrm{char}(K)$ does not divide any of the $\lambda_i$ or orders of the stabilizers, ensures that both the automorphism groups of points in $\mathcal{V}$ and their $\PGL_2$-stabilizers are tame, and thus so are their extensions giving the geometric stabilizers of $[\mathcal{V}/\PGL_2]$.
    
    %\cite[Proposition 2.5]{AOV}
    
    %These are extensions of the stabilizer groups of points in $\mathcal{V}$ (which are cyclic) and the stabilizer group of the $\PGL_2$-action, which by assumption is finite and reduced. This precisely guarantees the condition of being unramified and finishes the proof that $[\mathcal{V}/\PGL_2]$ is Deligne--Mumford.
    
    Finally, for geometric points $[\vec u] \in \mathcal{V} = \Hom_n(\Pb^1,\Pcv)$ we have by definition that $[\vec u]$ has no base points and, since the degree $n$ is positive, the map $\varphi_{[\vec u]} : \Pb^1 \to \Pcv$ is non-constant. Then it follows that it has finite, reduced and tame $\PGL_2$-stabilizer by Proposition \ref{Prop:finredstabilizercrit}.
    \end{proof}
    %For $n>0$, any element of $\Hom_n(\Pb^1,\Pcv)$ has a finite cyclic stabilizer, and thus $\M_{n}^{\lambdavec}$ is a smooth Deligne--Mumford stack. \jocomment{Add some references here, Proposition C.3 in \cite{Schmitt} together with action having finite stabilizers. }
    %https://www.math.uzh.ch/typo3conf/ext/qfq/Classes/Api/download.php?s=601d2907b9b73
    %https://www.homepages.ucl.ac.uk/~ucahgvi/other/stacks.pdf

    \medskip

    \subsection{Coarse moduli spaces via Geometric Invariant Theory}
    \label{sec:GIT}%\texorpdfstring{$[ \mathcal{P}(\vec{\Lambda}) / \PGL_2 ]$}{Mnlambda}
    
    \par We now apply the machinery of Geometric Invariant Theory to explicitly construct the coarse moduli space $M_{n}^{\lambdavec}$ of $\M_{n}^{\lambdavec}$ as well as suitable modular compactifications and corresponding coarse moduli spaces of them. For this, we use the natural inclusion $\Hom_n(\Pb^1,\Pcv) \subsetneq \Pov$ in the ambient weighted projective stack $\Pov$, analyze the stable and semistable points for the action of $\PGL_2$ on $\Pov$ and form their quotients.

    As GIT is formulated in the setting of a group action on a variety, however, we must be careful here, since in general $\Pov$ is a Deligne--Mumford stack. To remedy this, we note that the coarse moduli space $\Pb(\Lambdavec)$ of $\Pov$, which is a classical weighted projective space, has a  natural action of $\PGL_2$ compatible with the action of $\PGL_2$ on $\Pov$. %\jocomment{Sweep linearized line bundle under the rug?}
    Then we can use the Hilbert--Mumford criterion to analyze the (semi)stable loci for this action.

    \begin{prop} \label{Prop:stablesemistablecrit}
    A point $[\vec{u}] = [u_0: \ldots : u_N] \in \Pb(\Lambdavec)$ is semistable for the action of $\PGL_2$ if and only if for each $p \in \Pb^1$ there exists $0 \leq i \leq N$  such that
    \begin{equation} \label{eqn:ssinequality}
        \mathrm{ord}_p(u_i) \leq  \frac{n \lambda_i}{2} \,
    \end{equation}
    The point is stable if and only if for each $p \in \Pb^1$ there exists $0 \leq i \leq N$ such that the inequality \eqref{eqn:ssinequality} is strict. 
    In particular, the sets of stable and semistable points agree if and only if the degree $n$ and all numbers $\lambda_i$ are odd.
    \end{prop}

    \begin{proof}
    The proof uses the Hilbert--Mumford numerical criterion \cite[Theorem 2.1]{GIT}, following similar arguments as in \cite[Proposition 5.1]{Miranda} or \cite[Proposition 2.2]{Silverman}. Consider the one-parameter subgroup
    \[
    \Gb_m \to \PGL_2, t \mapsto \begin{pmatrix}
    t & 0 \\ 0 & t^{-1}
    \end{pmatrix}\,.
    \]
    For $u_i = a_{i,0} X^{n \lambda_i} + \ldots + a_{i,n\lambda_i} Y^{n \lambda_i}$, this subgroup acts with weight $n \lambda_i - 2j$ on the coordinate $a_{i,j}$. By the Hilbert--Mumford criterion\footnote{~The version of the criterion we use here can be obtained via \cite[Proposition 2.3]{GIT}. Just as in the proof of \cite[Proposition 5.1]{Miranda}, we need to embed our weighted projective space $\Pb(\Lambdavec)$ into a larger (unweighted) projective space and analyze stability there since the Proposition in \cite{GIT} is formulated for ordinary projective spaces. We leave the details to the interested reader.}, the point $[\vec{u}_0]$ is \emph{not} semistable if and only if for some point $[\vec{u}]$ in the orbit of $[\vec{u}_0]$, we have that for all $i,j$ either $a_{i,j} = 0$ or $n \lambda_i - 2j<0$. Equivalently, we have $a_{i,j}=0$ for $j \leq (n \lambda_i)/2$ or in other words $\mathrm{ord}_{[1:0]}(u_i) > (n \lambda_i)/2$. Since the orbit of $[1:0]$ under $\PGL_2$ is all of $\Pb^1$, we obtain the criterion formulated above. The argument for the stable points is similar, allowing simple inequalities instead of strict inequalities. In case that both degree $n$ and weight $\lambda_i$ are odd, the number $n \lambda_i /2$ is a half-integer, and thus the inequality \eqref{eqn:ssinequality} is satisfied if and only if it is strictly satisfied, rendering stable and semistable points to coincide.
    \end{proof}
    Denote by $\Pb(\Lambdavec)^s \subseteq \Pb(\Lambdavec)^{ss} \subseteq \Pb(\Lambdavec)$ the loci of (semi)stable points and denote by $\Pov^s \subseteq \Pov^{ss} \subseteq \Pov$ their preimages under the coarse moduli space morphism $\Pov \to \Pb(\Lambdavec)$. Then on the one hand, we write
    \begin{equation}
        M_{n}^{\lambdavec,s} = \Pb(\Lambdavec)^{s} \git \PGL_2 \subseteq M_{n}^{\lambdavec,ss} = \Pb(\Lambdavec)^{ss}\git \PGL_2\,,
    \end{equation}
    the GIT-quotients of the (semi)stable open subsets of $\Pb(\Lambdavec)$.
    On the other hand, we denote
    \begin{equation}
        \M_{n}^{\lambdavec,s} = [\Pov^{s}/\PGL_2] \subseteq \M_{n}^{\lambdavec,ss} = [\Pov^{ss}/\PGL_2]\,,
    \end{equation}
    the quotient stacks of the stacky (semi)stable loci by the $\PGL_2$-action (again using the notion of quotients from \cite{Romagny}). Note that by the criterion above, for $n \geq 1$ the open substack $\Hom_n(\Pb^1,\Pcv) \subseteq \Pov$ is contained inside the stable locus, so that $\M_n^{\lambdavec}$ is naturally an open substack of $\M_{n}^{\lambdavec,s}$.
    \begin{prop} \label{Pro:stablesemistablestackproperties}
    Assume that $\mathrm{char}(K) = 0$ or greater than \eqref{eqn:maxdegree}. Then $\M_{n}^{\lambdavec,s}$ is a smooth, separated, tame Deligne--Mumford stack. 
    On the other hand, we have that $\M_{n}^{\lambdavec,ss}$ is a smooth algebraic stack.
    % If $\mathrm{char}(K)=0$ and the stable and semistable locus coincide, we have that $\M_{n}^{\lambdavec,s} = \M_{n}^{\lambdavec,ss}$ is a smooth, proper, tame Deligne--Mumford stack.
    \end{prop}
    \begin{proof}
    The fact that $\M_{n}^{\lambdavec,s}$ and $\M_{n}^{\lambdavec,ss}$ are smooth algebraic stacks again follows from their construction as stack quotients by \cite[Theorem 4.1]{Romagny}. 
    To see the separatedness\footnote{~See \cite[Tag 04YV]{Stacks} for a reminder of the definition.} of $\M_{n}^{\lambdavec,s}$ we need to show that the diagonal map $\Delta$ for $\M_{n}^{\lambdavec,s}$ is proper. Looking at the fibre diagram
    \[
    \begin{tikzcd}
    \PGL_2 \times \Pov^{s} \arrow[d] \arrow[r,"\sigma"] & \Pov^{s} \times \Pov^{s} \arrow[d]\\
    \M_{n}^{\lambdavec,s} \arrow[r,"\Delta"] & \M_{n}^{\lambdavec,s} \times \M_{n}^{\lambdavec,s}
    \end{tikzcd}
    \]
    this can be checked by proving that the action map $\sigma$ is proper. By \cite[Converse 1.13]{GIT} this follows from the fact that $\Pov^{s}$ is the locus of (properly) stable points for a $\PGL_2$-linearized line bundle on $\Pov$.\footnote{~Formally the cited result only applies in the realm of schemes. However, using \cite[Theorem 3.10]{Vakil} together with the fact that the coarse moduli space map $\Pov^{s} \to \mathbb{P}(\Lambda)^{s}$ is proper, the properness of the action can be checked on coarse moduli spaces, where \cite[Converse 1.13]{GIT} applies.}
    
    Next, let us prove that $\M_{n}^{\lambdavec,s}$ is a tame Deligne--Mumford stack. Here, by Theorem~\ref{Thm:QuotisDM}, it suffices to show that $\PGL_2$ acts with finite, reduced and tame stabilizers at geometric points of $\Pov^{s}$.
    Say that $[\vec u] \in \Pov^{s}$ is a geometric point where this is not satisfied.
    Then by Corollary \ref{Cor:finredtamecrit}, replacing $[\vec u]$ by an element of its $\PGL_2$-orbit, we can assume that $[\vec u]$ is of the form
    \[u_i = a_i \cdot (X^e Y^f)^{\lambda_i/\ell}\,,\]
    where $[\vec a] \in \Pcv$, $\ell = \gcd(\lambda_i : i=0,\ldots, N \text{ s.t. }a_i \neq 0)$ and $e,f \geq 0$ such that $e + f = \ell \cdot n$.
    Applying the criterion for stability from Proposition \ref{Prop:stablesemistablecrit} at the points $p=0, \infty \in \Pb^1$ we see that there exist $i_0, i_\infty \in \{0, \ldots, N\}$ such that
    \begin{align*}
        \mathrm{ord}_0(u_{i_0}) = e \frac{\lambda_{i_0}}{\ell} & < \frac{\lambda_{i_0} n}{2}\,,\\
        \mathrm{ord}_\infty(u_{i_\infty}) = f \frac{\lambda_{i_\infty}}{\ell} & < \frac{\lambda_{i_\infty}n}{2}\,.
    \end{align*}
    Dividing the first and second equality by $\lambda_{i_0}$ and $\lambda_{i_\infty}$, respectively, and adding the two resulting inequalities, we obtain
    \[
    (e+f)\frac{1}{\ell} = \ell \cdot n \frac{1}{\ell} = n < n\,,
    \]
    which gives a contradiction. This shows that indeed all stable points of $\Pov$ have finite, reduced and tame stabilizers.
    %Then, by Proposition \ref{Prop:finredstabilizercrit} the point $[\vec u]$ has at most two distinct base points (which, via the $\PGL_2$-action, we can assume to be at $0, \infty$) and the underlying rational map $\varphi_{[\vec u]}$ is constant.
    \end{proof}
    % \jocomment{Why is $\M_{n}^{\lambdavec,ss}$ proper? Ok if stable equals semistable, since coarse moduli is projective and coarse moduli map is proper. Maybe imitate \cite{GIT}?} 
    We are now ready to prove Theorem~\ref{Thm:Moduli_GIT}:
    \begin{proof}[Proof of Theorem~\ref{Thm:Moduli_GIT}]
    Part a) of the theorem follows from \cite[Theorem 4.1]{Romagny} and the fact that $\Pov$ is irreducible and of finite type. Part b) was proven in Proposition \ref{Pro:stablesemistablestackproperties}. Here, the fact that the Deligne--Mumford property of $\M_n^{\lambdavec}$ only requires $\mathrm{char}(K)$ not dividing any of the $\lambda_i$ or $n$ follows by an application of Proposition \ref{Prop:finredstabilizercrit} : the geometric points in $\M_n^{\lambdavec}$ correspond to finite morphisms $\Pb^1 \to \Pcv$ and so the weaker assumption on the characteristic is sufficient to show that their stabilizer groups are \'etale.
    
    The existence and properties of the geometric and categorical quotients of the sets of (semi)stable points follow from \cite[Theorem 1.10]{GIT}, with the projectivity of $M_{n}^{\lambdavec,ss}$ following from the remark above \cite[Corollary 1.12]{GIT}. The fact that $M_{n}^{\lambdavec}$, $M_{n}^{\lambdavec,s}$ are coarse moduli spaces for $\M_{n}^{\lambdavec}$, $\M_{n}^{\lambdavec,s}$ follows from a variation of the arguments in \cite[Lemma C.4, Remark C.5]{Schmitt}.\footnote{~The statement here is basically that for a group action on a stack, the operation of taking a moduli space followed by a geometric quotient is equivalent to taking the stack quotient followed by a moduli space. The original argument in \cite{Schmitt} uses the notion of a good moduli space from \cite{Alper}, but to work better in positive characteristic (where $\PGL_2$ is not linearly reductive), this can be replaced by the classical notion of a coarse moduli space, i.e., being a bijection on geometric points and a universal map to algebraic spaces. Then the arguments in \cite{Schmitt} go through basically unchanged, replacing the cited results from \cite{Alper} by \cite[Proposition 9.1]{KM}.}

    % Here the assumption on the characteristic is needed to ensure that $\PGL_2$ is linearly reductive.
    
    % By Theorem~\ref{Thm:QuotisDM} both $\M_{n}^{\lambdavec,s}$ and $\M_{n}^{\lambdavec,ss}$ are smooth, Deligne--Mumford stacks.
    
    Finally, for part d) by Proposition \ref{Prop:stablesemistablecrit} the assumption that all $\lambda_i$ and $n$ are odd implies that stable and semistable points coincide. By \cite{KM} (or the modern reformulation \cite[Theorem 1.1]{Conrad2}) the coarse moduli space map $\M_n^{\lambdavec,ss} \to M_n^{\lambdavec,ss}$ is proper, and so the properness of $\M_n^{\lambdavec,ss}$ follows from the properness of the variety $M_n^{\lambdavec,ss}$.
    %in characteristic $\mathrm{char}(K)$, the claim that $\M_n^{\lambdavec,s}=\M_n^{\lambdavec,ss}$ is proper follows by combining the properties above: it is clearly finite type over $K$, it is separated by Proposition \ref{Pro:stablesemistablestackproperties} and the universal closedness over $K$ follows by combining the universal closedness of the coarse moduli space map to $M_n^{\lambdavec,ss}$ (see \cite[Main Property (1)]{Alper} together with the properness of the variety $M_n^{\lambdavec,ss}$.
    \end{proof}
    %\jocomment{\url{https://arxiv.org/pdf/0804.2242.pdf} page 3 for semistable locus?}
    \begin{rmk}
    It is an interesting question whether the stack $\M_{n}^{\lambdavec,ss}$ satisfies the existence part of the valuative criterion of properness even when the stable and semistable points do not coincide. However, since we do not need results in this direction, we will not pursue this question in the following.
    \end{rmk}

    \medskip

    \section{Motives \& Weighted point counts over finite fields}
    %Motive / Weighted point count of \texorpdfstring{$\M_{n}^{\lambdavec}$}{Mnlambda} over \texorpdfstring{$\Fb_q$}{FFq}
    \label{sec:Motive}

    Ekedahl in 2009 introduced the Grothendieck ring $K_0(\mathrm{Stck}_K)$ of algebraic stacks extending the classical Grothendieck ring $K_0(\mathrm{Var}_K)$ of varieties first defined by Grothendieck in 1964 in a letter to Serre. In \cite{Ekedahl}, every algebraic stack $\Xc$ of finite type over $K$ with affine stabilizers has the motivic class, i.e., the \emph{motive} of $\Xc$ as $\{\Xc\} \in K_0(\text{Stck}_K)$.

    \begin{defn}\label{defn:GrothringStck}
        \cite[\S 1]{Ekedahl}
        Fix a field $K$. Then the \emph{Grothendieck ring $K_0(\mathrm{Stck}_K)$ of algebraic stacks of finite type over $K$ all of whose stabilizer group schemes are affine} is an abelian group generated by isomorphism classes of algebraic stacks $\{\Xc\}$ modulo relations:
        \begin{itemize}
            \item $\{\Xc\}=\{\Zc\}+\{\Xc \setminus \Zc\}$ for $\Zc \subset \Xc$ a closed substack,
            \item $\{\Ec\}=\{\Xc \times \Ab^n \}$ for $\Ec$ a vector bundle of rank $n$ on $\Xc$.
        \end{itemize}
        Multiplication on $K_0(\mathrm{Stck}_K)$ is induced by $\{\Xc\}\{\Yc\}\coloneqq\{\Xc \times_K \Yc\}$. 
        There is a distinguished element $\Lb\coloneqq\{\A^1\} \in K_0(\mathrm{Stck}_K)$, called the \emph{Lefschetz motive}.
    \end{defn}

    We recall the definition of a weighted point count of an algebraic stack $\Xc$ over $\Fb_q$:

    \begin{defn}\label{def:wtcount}
        The weighted point count of $\Xc$ over $\Fb_q$ is defined as a sum:
        \[
        \#_q(\Xc)\coloneqq\sum_{x \in \Xc(\Fb_q)/\sim}\frac{1}{|\mathrm{Aut}(x)|},
        \]
        where $\Xc(\Fb_q)/\sim$ is the set of $\Fb_q$--isomorphism classes of $\Fb_q$--points of $\Xc$ (i.e., the set of non--weighted points of $\Xc$ over $\Fb_q$), and we take $\frac{1}{|\mathrm{Aut}(x)|}$ to be $0$ when $|\mathrm{Aut}(x)|=\infty$.
    \end{defn}
    
    A priori, the weighted point count can be $\infty$, but when $\Xc$ is of finite type, then the stratification of $\Xc$ by schemes as in \cite[Proof of Lemma 3.2.2]{Behrend} implies that $\Xc(\Fb_q)/\sim$ is a finite set, so that $\#_q(\Xc)<\infty$.
    As the Grothendieck ring $K_0(\mathrm{Stck}_K)$ of algebraic stacks is universal for additive invariants of the category $\mathrm{Stck}_K$ of algebraic stacks over $K$, it is easy to see that when $K=\Fb_q$, the point counting measure $\{\Xc\} \mapsto \#_q(\Xc)$ gives a well-defined ring homomorphism $\#_q: K_0(\mathrm{Stck}_{\Fb_q}) \rightarrow \Q$ (c.f. \cite[\S 2]{Ekedahl}). %rendering the weighted point count $\#_q(\Xc)$ of an algebraic stack $\Xc$ of finite type over $\Fb_q$. 

    \medskip

    Recall that an algebraic group $G$ is \textit{special} in the sense of \cite{Serre} and \cite{Grothendieck},  if every $G$-torsor is Zariski-locally trivial; for example $\GL_{d},~\SL_{d}$ are special and $\PGL_2,~\PGL_3$ are non-special. If $\mathcal X \rightarrow \mathcal Y$ is a $G$-torsor and $G$ is special, then we have $\{\mathcal X\} = \{G\} \cdot \{\mathcal Y\}$ (this is immediate when $\mathcal Y$ is a scheme, and it was shown by \cite[Proposition 1.1 iii)]{Ekedahl} when $\mathcal Y$ is an algebraic stack). In particular, applying this multiplicative relation to the universal torsor $\Spec~K \to \Bc G$ where $\Bc G$ is the classifying stack for the group $G$ we acquire the formula $\{G\}^{-1} = \{\Bc G\}$ for special groups as the motive of the base $\Spec~K$ is the multiplicative identity in $K_0(\mathrm{Stck}_{/K})$. Since many algebraic stacks can be written locally as a quotient of a scheme by an algebraic group $\Gb_m$, the following lemma is useful:
    
    \begin{lem}\label{lem:Gm_quot}
        For any $\Gb_m$-torsor $\mathcal X \rightarrow \mathcal Y$ of finite type algebraic stacks, we have $[\mathcal Y]=[\mathcal X][\Gb_m]^{-1}$.
    \end{lem}
    \begin{proof}
        This follows from \cite[Proposition 1.1 iii), 1.4]{Ekedahl} and the definition of $K_0^{\mathrm{Zar}}(\mathrm{Stck}_K)$ in \cite[\S 1]{Ekedahl}.
    \end{proof}

    Let us also recall the motive of $\SL_{2}$ in the Grothendieck ring of stacks $K_0(\mathrm{Stck}_{/K})$. 
    \begin{lem}\label{lem:SL_Motive}
    We have the motive $\{\SL_{2}\} = \Lb \cdot (\Lb^{2}-1)$ .
    \end{lem}
    \begin{proof}
        Note that the determinant $\GL_{d} \to \Gb_m$ is an $\SL_{d}$-torsor. As $\SL_{d}$ is special, by the multiplicative relation for motives we have $\{\SL_{d}\} = \frac{\{\GL_{d}\}}{\{\Gb_m\}} \in K_0(\mathrm{Stck}_{/K})$.
        By \cite[Proposition 1.1 i)]{Ekedahl}, we know
        %Bruhat decomposition
        \begin{align*}
            \{\GL_{d}\} &= \prod\limits_{i=0}^{d-1} (\Lb^{d}-\Lb^{i}) = (\Lb^{d}-1)(\Lb^{d}-\Lb) \cdots (\Lb^{d}-\Lb^{d-1})\,,\\
            \{\Gb_m\} &= \{\GL_{1}\} = (\Lb-1)\,.
        \end{align*}
        Thus we have $\{\SL_{d}\} = (\Lb-1)^{-1} \cdot \prod\limits_{i=0}^{d-1} (\Lb^{d}-\Lb^{i})$ which in the case of $d=2$ renders $\{\SL_{2}\} = \Lb \cdot (\Lb^{2}-1)$ .
    \end{proof}
    Since the map $\GL_2 \to \PGL_2$ is a $\Gb_m$-torsor, we can apply Lemma \ref{lem:Gm_quot} directly to conclude that 
    $$\{\PGL_2\} = \{\GL_2\}\{\Gb_m\}^{-1}=\{\SL_2\}\,.$$
    Thus the motive of the special group $\SL_{2}$ is equal to the motive of the non-special group $\PGL_{2}$. For a non-special and connected group such as $\PGL_{2}$, it is not obvious that the motive of $\Bc \PGL_{2}$ is the inverse of the motive of the group $\PGL_{2}$ and thus we have the expected formula $\{\PGL_{2}\}^{-1} = \{\Bc \PGL_{2}\} = \frac{1}{\Lb \cdot ( \Lb^{2}-1 )}$ from the multiplicative relation for torsors. This is shown to be true by the work of \cite{Bergh}.

    \begin{thm}[\cite{Bergh} Theorem A.]\label{Thm:Bergh_PGL}
    Let $K$ be a field of characteristic not equal to 2. Then the class of the classifying stack $\{\Bc \PGL_{2}\}$ is the inverse of the class of $\PGL_{2}$ in $K_0(\mathrm{Stck}_{/K})$.
    \end{thm}
    
    Now we turn to the computation of the motive of $\M_{n}^{\lambdavec}$ in the Grothendieck ring of stacks. The following result gives an explicit formula for this motive.
    
       \begin{thm}\label{Thm:motivequotient1}
    Let $K$ be a field let the degree $n \in \Zb_{\geq 1}$ be an odd number. For $\vec{\lambda} = (\lambda_0, \dotsc, \lambda_N) \in \mathbb{Z}_{\geq 1}^{N+1}$ let $|\vec{\lambda}|\coloneqq\sum\limits_{i=0}^{N} \lambda_i$. 
    Then we have %  the motive $\{\M_{n}^{\lambdavec}\} \in K_0(\mathrm{Stck}_{K})$ is a polynomial in $\Lb$ (i.e., mixed Tate motive)
    \begin{equation} \label{eqn:Mlambdamotiveeqn}
        \left\{\M_{n}^{\lambdavec}\right\} = \frac{\left\{\Hom_n(\Pb^1,\Pcv)\right\} }{\left\{\PGL_2\right\}}
    \end{equation}
    which has the explicit formula
    \[
       \left\{\M_{n}^{\lambdavec}\right\} =
        \begin{cases}
            \Lb^{|\vec{\lambda}|n - N - 1} \cdot \left( \Lb^{N-1} + \Lb^{N-3} + \dotsc + \Lb^{2} + 1 \right) \cdot \left( \Lb^{N-1} + \dotsc + 1 \right), & \mathrm{~if~N~odd.} \\\\
            \Lb^{|\vec{\lambda}|n - N - 1} \cdot \left( \Lb^{N} + \dotsc + 1 \right) \cdot \left( \Lb^{N-2} + \Lb^{N-4} + \dotsc + \Lb^{2} + 1 \right),  & \mathrm{~if~N~even.}
        \end{cases}
    \]
    % Similarly for ambient $\left[\Pov / \PGL_2\right]$, we acquire its motive as a rational function in $\Lb$ 
    % \begin{align*}\label{eqn:Pov_L}
    %     \left\{ \left[\Pov / \PGL_2\right] \right\} &= \frac{\Lb^{|\vec{\lambda}|n+N} + \Lb^{|\vec{\lambda}|n+N-1} + \dotsc + \Lb + 1}{\Lb \cdot ( \Lb^{2}-1 )} = \frac{\Lb^{|\vec{\lambda}|n+N+1}-1}{\Lb(\Lb-1)(\Lb^2-1)} \\\\
    %     &= \Lb^{|\vec{\lambda}|n+N-3} + \Lb^{|\vec{\lambda}|n+N-4} + 2 \Lb^{|\vec{\lambda}|n+N-5} + 2 \Lb^{|\vec{\lambda}|n+N-6} + O(\Lb^{|\vec{\lambda}|n+N-7}).
    % \end{align*}\\
    %where $\Lb^1\coloneqq\left\{\Ab^1_K\right\}$ is the Lefschetz motive. 
    This implies that for the weighted point counts over $\Fb_q$ we have
    \[
        \#_q\left(\M_{n}^{\lambdavec}\right) =
        \begin{cases}
            q^{|\vec{\lambda}|n - N - 1} \cdot \left( q^{N-1} + q^{N-3} + \dotsc + q^{2} + 1 \right) \cdot \left( q^{N-1} + \dotsc + 1 \right), & \mathrm{~if~N~odd.} \\\\
            q^{|\vec{\lambda}|n - N - 1} \cdot \left( q^{N} + \dotsc + 1 \right) \cdot \left( q^{N-2} + q^{N-4} + \dotsc + q^{2} + 1 \right),  & \mathrm{~if~N~even.}
        \end{cases}
    \]
    %\medskip
    % \begin{align*}\label{eqn:Pov_q}
    %     \#_q\left(\left[\Pov / \PGL_2\right]\right) &= \frac{q^{|\vec{\lambda}|n+N} + q^{|\vec{\lambda}|n+N-1} + \dotsc + q + 1}{q \cdot ( q^{2}-1 )} = \frac{q^{|\vec{\lambda}|n+N+1}-1}{q(q-1)(q^2-1)} \\\\
    %     &= q^{|\vec{\lambda}|n+N-3} + q^{|\vec{\lambda}|n+N-4} + 2 q^{|\vec{\lambda}|n+N-5} + 2 q^{|\vec{\lambda}|n+N-6} + O(q^{|\vec{\lambda}|n+N-7}).
    % \end{align*}
    \end{thm}
    
    When trying to prove the result above, we encounter a problem: since $\PGL_2$ is not a special group, the multiplicativity relation for the universal $\PGL_2$-torsor over the classifying stack $\{\Bc \PGL_{2}\}$ need not imply the same relation for arbitrary $\PGL_2$-torsors and thus it is not a priori clear that the motive of the quotient stack $\M_{n}^{\lambdavec}$ equals the motive of the pre-quotient $\Hom_n(\Pb^1,\Pcv)$ divided by the motive of $\PGL_2$. To overcome this difficulty, we show that in many cases, there exists an alternative representation of $\M_{n}^{\lambdavec}$ as a quotient stack for the special group $\GL_2$. 

    \medskip
    
    \par To set the stage for this, recall the identification
    \[
    U_\Lambdavec = \bigoplus_{i=0}^N H^0\left(\Pb^1, \mathcal{O}_{\Pb^1}(n \lambda_i)\right)
    \]
    from above. The action \eqref{eqn:matrixactioncoordinates} of $\PGL_2$ on the homogeneous coordinates $X,Y$ is induced by an action of $\GL_2$ on the spaces $H^0\left(\Pb^1, \mathcal{O}_{\Pb^1}(n \lambda_i)\right)$, which can be interpreted as the spaces of homogeneous polynomials of degrees $n \lambda_i$. 
    
    Now it is almost, but not quite, true that we can see $\M_{n}^{\lambdavec}$ as an open substack of $[U_\Lambda / \GL_2]$. Indeed, 
    consider the diagonal subgroup $T=\Gb_m \subseteq \GL_2$ inside $\GL_2$. Then from the action \eqref{eqn:matrixactioncoordinates} on the coordinates \eqref{eqn:coordinaterepr} we see that $T$ acts with weight vector\footnote{~Formally, since our action was defined by $f \mapsto f \circ \varphi^{-1}$ to obtain a left action, we would get weight $-n \Lambdavec$ in standard conventions. To avoid unnecessary clutter in the notation, we suppress these sign issues, which can anyway be resolved by applying the automorphism $t \mapsto t^{-1}$ of the diagonal torus $T$.} $n \Lambdavec$ on $U_\Lambdavec$. For $n=1$ we thus have 
    \[
    \left[(U_\Lambdavec \setminus \{0\}) / T\right] = \Pov\,,
    \]
    containing $\Hom_n(\Pb^1,\Pcv)$ as an open substack by \eqref{eqn:HomsubsetPLambda}. Moreover, the induced action of $\PGL_2 = \GL_2 / T$ on $\Pov$ precisely restricts on $\Hom_n(\Pb^1,\Pcv)$ to the action used to define $\M_{n}^{\lambdavec}$. Therefore, by \cite[Remark 2.4]{Romagny} we have
    \begin{align} 
        \M_{n}^{\lambdavec} &= \left[\Hom_n(\Pb^1,\Pcv) / \PGL_2\right] \nonumber \\ &\subseteq \left[\Pov / \PGL_2\right] = \left[[(U_\Lambdavec \setminus \{0\}) / T] / (\GL_2 / T)\right] = \left[(U_\Lambdavec \setminus \{0\}) / \GL_2\right]\,.\label{eqn:quotientswap}
    \end{align}
    Thus at least for $n=1$ we can always see $\M_{n}^{\lambdavec}$ as an open substack of a $\GL_2$-quotient stack. 
    What prevents us from pursuing this strategy for arbitrary degree $n$ and weight $\lambdavec = (\lambda_0, \dotsc, \lambda_N)$ is the fact that in general, the torus $T$ does not act with the correct weights on $U_\Lambdavec$ to produce the quotient $\Pov$. However, in many cases, this can be fixed by modifying the original $\GL_2$-action on $U_\Lambdavec$.
    Indeed, for the natural one-dimensional representation
    \[
    \mathsf{det} : \GL_2 \to \Gb_m
    \]
    given by the determinant, the restriction to $T$ gives the representation $T \to \Gb_m, t \mapsto t^2$ of weight $2$.% Assume that
    % \begin{equation} \label{eqn:evenness1}
    %     (n-1) \cdot \lambda_i \text{ is even for all }i \iff n \text{ is odd or }(\lambda_i \text{ is even for all }i)\,.
    % \end{equation}
    %$n$ is odd. 

    Assume that $n$ is odd. Then we consider the action of $\GL_2$ on $U_\Lambdavec$ obtained from the standard action by tensoring with $\mathsf{det}^{-(n-1) \lambda_i/2}$ on the summand $H^0\left(\Pb^1, \mathcal{O}_{\Pb^1}(n \lambda_i)\right)$. The diagonal torus $T$ acts with weight $\Lambdavec = n \Lambdavec - 2 \cdot (n-1)/2 \Lambdavec$ as desired and it is still true that the induced action of $\PGL_2$ on $\Pov$ is the classical one. Indeed, the new action of $[A] \in \PGL_2$ on $[(u_i(X,Y))_{i=0}^N] \in \Pov$ is given by
    \begin{align*}
        \left[\left((\mathsf{det}(A)^{-(n-1)/2})^{\lambda_i} \cdot u_i(A^{-1} \cdot (X,Y))\right)_{i=0}^N\right] = \left[\left(u_i(A^{-1} \cdot (X,Y))\right)_{i=0}^N\right]  \in \Pov\,,
    \end{align*}
    which agrees with the old action. For this proof, \textit{it is crucial that $n$ is odd}. 

    Note that when all $\lambda_i$ are even, we \emph{can} define the modified action even for $n$ divisible by $2$, since $(n-1) \lambda_i/2$ is an integer. However, in this case it is not clear that the induced action of $\PGL_2$ coincides with the original action. 

    %See also Question \ref{Quest:strangeaction} for related considerations.
    Returning to the case of $n \in \N$ odd, we can then repeat the proof explained in \eqref{eqn:quotientswap} for $n=1$ and prove both Main Theorem~\ref{Thm:WeierstrassMotive} and Theorem~\ref{Thm:motivequotient1}:

    %Finally, our computations of motives and point counts presented in Main Theorem \ref{Thm:WeierstrassMotive} can be generalized for $\M_n^{\lambdavec}$.

    %\par We consider the motive of the moduli stack $\M_{n}^{\lambdavec}$ that is an open substack of $\left[\Pov / \PGL_2\right]$ (i.e., the $\PGL_2$ stack quotient on the ambient weighted projective stack $\Pov$) in the Grothendieck ring $K_0(\mathrm{Stck}_{K})$ of $K$--stacks which provides the weighted point counts of the moduli stacks $\M_{n}^{\lambdavec}$ and $\left[\Pov / \PGL_2\right]$ over $\Fb_q$. 

    \begin{proof}[Proof of Main Theorem~\ref{Thm:WeierstrassMotive} and Theorem~\ref{Thm:motivequotient1}]
    For the quotient map $U_\Lambdavec \setminus \{0\} \to \Pov$, let $~\widehat U \subset U_\Lambdavec \setminus \{0\}$ be the preimage of $\Hom_n(\Pb^1,\Pcv) \subseteq \Pov$. Then we claim that there exists an isomorphism \\
    \begin{equation} \label{eqn:Mlambdaisom}
        \M_{n}^{\lambdavec} \cong \left[\widehat U / \GL_2\right]
    \end{equation}
    Indeed, this follows from the construction of the modified group action together with \cite[Remark 2.4]{Romagny} as explained in \eqref{eqn:quotientswap}. For the equality \eqref{eqn:Mlambdamotiveeqn} note first that $\widehat U \to \Hom_n(\Pb^1,\Pcv)$ is a $\Gb_m$-torsor and thus, since $\Gb_m$ is special, we have
    \[
    \left\{\widehat U\right\} = \left\{\Hom_n(\Pb^1,\Pcv)\right\} \cdot \left\{\Gb_m\right\}\,.
    \]
    Then with the isomorphism \eqref{eqn:Mlambdaisom} we conclude
    \begin{align*}
    \left\{\M_{n}^{\lambdavec}\right\} &= \frac{\left\{\widehat U\right\}}{\left\{\GL_2\right\}} = \frac{\left\{\Hom_n(\Pb^1,\Pcv)\right\} \cdot \left\{\Gb_m\right\}}{\left\{\GL_2\right\}} = \frac{\left\{\Hom_n(\Pb^1,\Pcv)\right\}}{\frac{\left\{\GL_2\right\}}{\left\{\Gb_m\right\}}} \\ &= \frac{\left\{\Hom_n(\Pb^1,\Pcv)\right\} }{\left\{\PGL_2\right\}}\,,
    \end{align*}
    by using that $\GL_2$ is special for the first equality. 

    Since the motive $\left\{\Hom_n(\Pb^1,\Pcv)\right\} = \left(\sum\limits_{i=0}^{N} \Lb^{i}\right) \cdot \left(\Lb^{|\vec{\lambda}|n}-\Lb^{|\vec{\lambda}|n - N}\right)$ by \cite[Proposition 4.5.]{HP2}, we have
    \begin{align*}
        \left\{\M_{n}^{\lambdavec}\right\} &= \frac{\left\{\Hom_n(\Pb^1,\Pcv)\right\}}{\left\{\PGL_2\right\}}\\
        &=\frac{\left(\sum\limits_{i=0}^{N} \Lb^{i}\right) \cdot \left(\Lb^{|\vec{\lambda}|n}-\Lb^{|\vec{\lambda}|n - N}\right)}{\Lb \cdot ( \Lb^{2}-1 )} \\
        &=\frac{\left( \Lb^{N} + \dotsc + 1 \right) \cdot \Lb^{|\vec{\lambda}|n - N} \cdot \left( \Lb^{N} - 1 \right)}{\Lb \cdot ( \Lb^{2}-1 )} \\
        &=\frac{\Lb^{|\vec{\lambda}|n - N - 1} \cdot \left( \Lb^{N} + \dotsc + 1 \right) \cdot \left( \Lb^{N-1} + \dotsc + 1 \right)}{ \Lb + 1 } \\
        &=\begin{cases}
            \Lb^{|\vec{\lambda}|n - N - 1} \cdot \left( \Lb^{N-1} + \Lb^{N-3} + \dotsc + \Lb^{2} + 1 \right) \cdot \left( \Lb^{N-1} + \dotsc + 1 \right) &\mathrm{,~if~N~odd.} \\
            \Lb^{|\vec{\lambda}|n - N - 1} \cdot \left( \Lb^{N} + \dotsc + 1 \right) \cdot \left( \Lb^{N-2} + \Lb^{N-4} + \dotsc + \Lb^{2} + 1 \right)   &\mathrm{,~if~N~even.} 
        \end{cases}
    \end{align*}

    We note that $\left\{\Pov\right\} = \sum\limits_{i=0}^{|\vec{\lambda}|n+N} \Lb^{i}$ and by a similar argument as above.

    \begin{align*}
        \left\{ \left[\Pov / \PGL_2\right] \right\} &= \frac{\left\{\Pov\right\}}{\left\{\PGL_2\right\}}\\
        &=\frac{\Lb^{|\vec{\lambda}|n+N} + \Lb^{|\vec{\lambda}|n+N-1} + \dotsc + \Lb + 1}{\Lb \cdot ( \Lb^{2}-1 )}
    \end{align*}
    The specialization of these formulas to $\lambdavec=(4,6)$ combined with Main Theorem \ref{Thm:Weierstrass_Stack_Quotient} also concludes the proof of Main Theorem~\ref{Thm:WeierstrassMotive}.
    \end{proof}

    Below we compute some low-dimensional examples when $N = 1,2,3$.
    \begin{align*}
    \left\{\M_{n}^{(\lambda_0, \lambda_1)}\right\} &= \frac{ \left\{\Hom_n(\Pb^1,\Pc(\lambda_0, \lambda_1))\right\}}{ \left\{\PGL_2\right\}} &&= \Lb^{|\vec{\lambda}|n - 2} \\
    \left\{\M_{n}^{(\lambda_0, \lambda_1, \lambda_2)}\right\} &= \frac{ \left\{\Hom_n(\Pb^1,\Pc(\lambda_0, \lambda_1, \lambda_2))\right\}}{ \left\{\PGL_2\right\}} &&= \Lb^{|\vec{\lambda}|n - 3} \cdot \left( \Lb^{2} + \Lb + 1 \right) \\
    \left\{\M_{n}^{(\lambda_0, \lambda_1, \lambda_2, \lambda_3)}\right\} &= \frac{ \left\{\Hom_n(\Pb^1,\Pc(\lambda_0, \lambda_1, \lambda_2, \lambda_3))\right\}}{ \left\{\PGL_2\right\}} &&= \Lb^{|\vec{\lambda}|n - 4} \cdot \left( \Lb^{4} + \Lb^{3} + 2 \Lb^{2} + \Lb + 1 \right) 
    \end{align*}

    %We note that $\{\Spec(K)\} = 1$, $\{\Pb^2\} = \Lb^{2} + \Lb + 1$ and $\{\mathrm{Gr(2,4)}\} = \Lb^{4} + \Lb^{3} + 2 \Lb^{2} + \Lb + 1$ where $\mathrm{Gr(2,4)}$ stands for the Grassmannian of 2-planes in the 4-space that is a smooth quadric hypersurface in $Pb^5$. The motive $\{\mathrm{Gr(2,4)}\}$ follows from the well-known Schubert cell decomposition. 

    \medskip

    Since $\SL_2$ is special, it is immediate that
    \[
    \left\{ [\Hom_n(\Pb^1,\Pcv) / \SL_2] \right\} = \frac{\{\Hom_n(\Pb^1,\Pcv)\}}{\{\SL_2\}}\,,
    \]
    and since $\{\SL_2\} = \{\PGL_2\}$ by Lemma~\ref{lem:SL_Motive}, this class agrees with the class of $\M_{n}^{\lambdavec}$ by Theorem~\ref{Thm:motivequotient1}. The following question is thus quite natural:
    \begin{ques}\label{ques:Motive_Even}
    Do we have equalities of motives or weighted point counts
    $$\left\{ [\Hom_n(\Pb^1,\Pcv) / \PGL_2] \right\} = \left\{ [\Hom_n(\Pb^1,\Pcv) / \SL_2] \right\}$$
    $$\#_q\left( [\Hom_n(\Pb^1,\Pcv) / \PGL_2] \right) = \#_q\left( [\Hom_n(\Pb^1,\Pcv) / \SL_2] \right)$$ 
    even when the degree $n$ is divisible by two? 
    \end{ques}

    Finally, as discussed in the introduction, the stack $\M_{n}^{\lambdavec}$ for $\lambdavec=(1,3)$ potentially has a modular interpretation as the moduli stack of stable Weierstrass fibrations with $[\Gamma_1(3)]$-level structure by an isomorphism $\Me[\Gamma_1(3)] \iso \Pc(1,3)$ over $\Spec(\Zb[1/ 3])$ as in \cite[Proposition 4.5]{HMe}. Below we collect the nice properties of this stack and its smooth modular compactification via GIT, as well as its motive induced weighted point count over $\Fb_q$ obtained by combining the theorems of the last two sections.

    \begin{thm} \label{Thm:Level_SP}
    Fix an odd degree $n$ and a base field $K$ with $\mathrm{char}(K) \neq 3$ and $\nmid n$. Then the moduli stack $\WeiB_{n}^{(1,3)} \coloneqq \M_{n}^{(1,3), s} = \M_{n}^{(1,3), ss}$ is a smooth, proper and irreducible Deligne--Mumford stack of dimension $4n-2$ which admits a projective coarse moduli space. It is also tame for $\mathrm{char}(K)=0$ or $>3n$. 
    Moreover, from the inclusions
    \[
    \M_{n}^{(1,3)} \subsetneq \WeiB_{n}^{(1,3)} \subsetneq \left[\Pov / \PGL_2\right]\,
    \]
    we then obtain the following bounds
    \begin{align*}
    q^{4n - 2} = \#_q\left(\M_{n}^{(1,3)}\right) \leq \#_q\left(\WeiB_{n}^{(1,3)}\right) & \leq \#_q\left(\left[\Pov / \PGL_2\right]\right) = \frac{q^{4n+2}-1}{q(q-1)(q^2-1)} \\\\
    & = q^{4n-2} + q^{4n-3} + 2 q^{4n-4} + 2 q^{4n-5} + O(q^{4n-6}) \,.    
    \end{align*}
    \end{thm}

    \begin{proof}
    For an odd degree $n$, as all the modular weights $\lambdavec=(1,3)$ are odd numbers, we naturally acquire the geometric properties of $\WeiB_{n}^{(1,3)}$ as a consequence of Theorem \ref{Thm:Moduli_GIT}. The bounds naturally follow the specialization of the formulas in Theorem \ref{Thm:motivequotient1} to $\lambdavec=(1,3)$.
    \end{proof}

    \begin{comment}

    \begin{rmk}
    Note that $\Bc \SL_2 \to \Bc \PGL_2$ is a $\mu_2$-gerbe and this implies, through the base change along the classifying maps, that $[\Hom_n(\Pb^1,\Pcv) / \SL_2] \to [\Hom_n(\Pb^1,\Pcv) / \PGL_2]$ is a $\mu_2$-gerbe as well.

    Also we have the expected formula $\{\mu_{2}\}^{-1} = \{\Bc \mu_{2}\} = 1$ as follows:

    \begin{prop}[\cite{Bergh} Proposition 2.10.]\label{Prop:Bergh_mu}
    The class of the classifying stack $\Bc \mu_d$ for the group of $d$th roots of unity is trivial in $K_0(\mathrm{Stck}_{/K})$ for any field $K$. 

    \textbf{Interpreted as a $\mu_2$-torsor, we have the following equality of motives by the multiplicativity relation for torsors}
    $$\left\{ [\Hom_n(\Pb^1,\Pcv) / \PGL_2] \right\} = \{\mu_{2}\}^{-1} \cdot \left\{ [\Hom_n(\Pb^1,\Pcv) / \SL_2] \right\}$$

    and as $\{\mu_{2}\}^{-1} = \{\Bc \mu_{2}\} = 1$ by Proposition~\ref{Prop:Bergh_mu} the answer to the Question is positive.

    \end{prop}

    \end{rmk}

    \end{comment}

    \section{Moduli stacks of Weierstrass fibrations over the projective line} \label{Sect:Weierstrass} %\texorpdfstring{$\Weier_{n}$}{W 1,12\coloneqq[Homn(P1,Mbar 1,1)/PGL2]} 

    \subsection{Fibrations over the parameterized projective line} \label{Sect:Fibparametrized}
    Here we define the fine moduli stack $\WeierP_{n}$ of Weierstrass fibrations over parameterized $\Pb^1$ with discriminant degree $12n$ and a section and show that its components are isomorphic to open substacks of a weighted projective stack $\Pov$. Our treatment follows \cite{Miranda, Seiler}, see also \cite{HL} for related discussions. 
    The essential difference to previous accounts is that instead of working on the level of moduli functors and coarse moduli spaces, we provide a modern description in the language of algebraic stacks, in particular taking into account the automorphism groups of Weierstrass fibrations. 
    Throughout the discussion, we work over a field $K$ of characteristic $\mathrm{char}(K) = 0$ or $\neq 2,3$.
    
    We recall the definition of a Weierstrass fibration which originally goes back to the classical work of Artin and Swinnerton-Dyer \cite{ASD}. We present it here in the special case where the base is the smooth projective line: a \emph{Weierstrass fibration} over $\Pb^1$ is a pair $(f: X \to \Pb^1, s:\Pb^1 \to X)$ of
    \begin{itemize}
        \item  a flat, proper map $f$ from a reduced, irreducible scheme $X$ to $\Pb^1$, such that every geometric fibre is either an elliptic curve, a rational curve with a node or a rational curve with a cusp, and such that the generic fibre is smooth,
        \item a section $s$ of $f$ not passing through the singular points of the fibres of $f$.
    \end{itemize}
    Recall that the fibration is called \emph{minimal} if it is the Weierstrass model of a smooth minimal elliptic surface over $\Pb^1$ with a section (see \cite[Section III.3]{Miranda2} for details). Moreover, we say that the fibration is \emph{stable} if all geometric fibres $X_t$ for $t \in \Pb^1$ are stable curves (i.e., have at worst nodal singularities).

    \par Now we generalize this definition to families, to define the fine moduli stack of Weierstrass fibrations over a parameterized $\Pb^1$.

    \begin{defn} \label{Def:WeierP_Fixed}
        Given a base scheme $S$, a family of Weierstrass fibrations to $\Pb^1$ over $S$ is specified by the data
        \begin{equation} \label{eqn:WeierPfamily}
        \mathcal{X} \xrightarrow{f} S \times \Pb^1 \to S,\ S \times \Pb^1 \xrightarrow{s} \mathcal{X}
        \end{equation}
        such that 
        \begin{itemize}
            \item $f$ is a flat, proper map with section $s$, 
            \item the fibres $(\mathcal{X}_t \xrightarrow{f} \Pb^1, \Pb^1 \xrightarrow{s} \mathcal{X}_t)$ over geometric points $t$ in $S$ are Weierstrass fibrations.
        \end{itemize}
        Given a morphism $T \to S$, we have a natural pullback fibration
        \[
        \mathcal{X} \times_S T = \mathcal{X}_T \xrightarrow{f_T} T \times \Pb^1 \to T, T \times \Pb^1 \xrightarrow{s_T} \mathcal{X}_T\,.
        \]
        We say that two fibrations $\mathcal{X}, \mathcal{X}'$ over $S$ are isomorphic if there exists an isomorphism $\varphi: \mathcal{X} \to \mathcal{X}'$ making the following diagram commute:
        \begin{equation} \label{eqn:WeierPdiagram}
            \begin{tikzcd}
            \mathcal{X} \arrow[d,"f"] \arrow[r,"\varphi"] & \mathcal{X}' \arrow[d,"f'"]\\
            S\times \Pb^1 \arrow[r,"\mathrm{id}_{S \times \Pb^1}"] \arrow[u,bend left,"s"] & S\times \Pb^1\arrow[u,bend left,"s'"]
            \end{tikzcd}
        \end{equation}
        Denote by $\WeierP$ the moduli stack
        %\footnote{\jocomment{Do we want to show ourselves that this is an algebraic stack? This would require showing that families glue (stack) and that we have a cover by a scheme (algebraic stack). I think we can kind of avoid this since we show that the components of this thing are isomorphic as categories fibred in groupoids to things which are stacks.}}
        whose objects over a scheme $S$ are the families of Weierstrass fibrations to $\Pb^1$ over $S$ as defined above. The morphisms over $T \to S$ are given by pullback diagrams as defined above.
    \end{defn}
    We emphasize that in our definition above, the base of the Weierstrass fibration is a parameterized $\Pb^1$, and the isomorphisms \eqref{eqn:WeierPdiagram} are required to restrict to the identity on this $\Pb^1$. However, in preparation of the discussion below we remark that the action of the group $\PGL_2$ on $\Pb^1$ induces an action (in the sense of \cite{Romagny}) of $\PGL_2$ on $\WeierP$, where a section $g : S \to \PGL_2$ sends the family \eqref{eqn:WeierPfamily} to 
    \[
    \mathcal{X}  \xrightarrow{(\mathrm{id}_S \times g) \circ f} S \times \Pb^1 \to S,\ S \times \Pb^1 \xrightarrow{s \circ (\mathrm{id}_S \times g^{-1})} \mathcal{X}\,.
    \]
    
    %\jocomment{What about being a minimal Weierstrass model? The paper %\url{https://arxiv.org/pdf/1808.03539.pdf} Section 3 refers to lecture notes of Miranda saying that we should ask for duVal singularities of $X$ to obtain minimal Weierstrass fibrations.    Update: Comment how going to a minimal Weierstrass model corresponds to removing a base point of the rational map $\Pb^1 \dashrightarrow \mathcal{P}(4,6)$.}
    
    % \jocomment{Maybe \url{https://arxiv.org/pdf/2101.12689.pdf} Section 2 is also a useful resource}

    Given a Weierstrass fibration $f:X \to \Pb^1$, we recall the \emph{fundamental line bundle} (see \cite[Definition II.4.1]{Miranda2})
    \begin{equation} \label{eqn:fundlinebundpoint}
        L = \left(R^1f_* \mathcal{O}_X \right)^\vee \in \mathrm{Pic}(\Pb^1)\,.
    \end{equation}
    Let $n = \mathrm{deg}(L)$, then of course there exists an isomorphism $L \cong \mathcal{O}_{\Pb^1}(n)$. 
    
    The fundamental line bundle generalizes to families $f : \mathcal{X} \to S \times \Pb^1$. Indeed, we can set
    \begin{equation} \label{eqn:fundlinebundfam}
        \mathcal{L} = \left(R^1f_* \mathcal{O}_X \right)^\vee \in \mathrm{Pic}(S \times \Pb^1)\,.
    \end{equation}
    For $S$ connected we have $\mathrm{Pic}(S \times \Pb^1) = \mathrm{Pic}(S) \times \mathrm{Pic}(\Pb^1)$ (see \cite[Exercise 28.1.K]{Vakil}), so there exists $n$ such that
    \[
    \mathcal{L} = \pi_S^*(\mathcal{L}_0) \otimes \mathcal{O}_{\Pb^1}(n)\,.
    \]
    In particular this shows that the number $n$ is an invariant of the connected components of the stack $\WeierP$, so this stack decomposes into a disjoint union of open and closed substacks $\WeierP_n$ according to this degree.
    
    Fixing $n \geq 1$ we claim that there is a natural morphism
    \begin{equation} \label{eqn:Phimapdef}
        \Phi : \WeierP_n \to \Pov
    \end{equation}
    where $\lambda=(4,6)$ and 
    \[
    \Lambdavec = (\underbrace{4, \ldots, 4}_{4n+1 \text{ times}}, \underbrace{6, \ldots, 6}_{6n+1 \text{ times}} )\in \mathbb{Z}^{10n+2} \,.
    \]
    To construct $\Phi$, we need to repeat the arguments presented in \cite[Section II.5]{Miranda2}) in families. These arguments go through verbatim, replacing the fixed curve $C = \Pb^1$ with the base $S \times \Pb^1$ to show the following: given the Weierstrass fibration \eqref{eqn:WeierPfamily} and forming the line bundle $\mathcal{L}$ from \eqref{eqn:fundlinebundfam} there exist so-called \emph{Weierstrass data} $(A,B)$ given by
    \begin{equation} \label{eqn:ABdef}
        A \in H^0(S \times \Pb^1, \mathcal{L}^{\otimes 4}),\ B \in H^0(S \times \Pb^1, \mathcal{L}^{\otimes 6})\,,
    \end{equation}
    unique up to the action of $\lambda \in H^0(S \times \Pb^1, \mathcal{O})^{*} =  H^0(S, \mathcal{O})^{*}$ defined by $$\lambda \cdot (A,B) = (\lambda^4 A, \lambda^6 B)\,.$$ 
    It is characterized by the property that the discriminant section $\Delta = 4 A^3 + 27 B^2$ of $\mathcal{L}^{\otimes 12}$ does not vanish on any of the fibres of $S \times \Pb^1 \to S$ and the property that $(A,B)$ give the universal coefficients of the Weierstrass equation for $f$. 
    To describe this latter property, note that the sheaf $\mathcal{F} = f_* \mathcal{O}_{\mathcal{X}}(3 \cdot s(S))$ is a rank $3$ vector bundle on $S \times \Pb^1$ and the natural surjection
    \[
    f^* f_* \mathcal{O}_{\mathcal{X}}(3 s(S)) \to \mathcal{O}_{\mathcal{X}}(3 s(S))
    \]
    induces an embedding
    \[
    \begin{tikzcd}
    \mathcal X \arrow[r,hookrightarrow,"i"] \arrow[rd,"f",swap] & \Pb(\mathcal{F}) \arrow[d,"p"]\\
    & S \times \Pb^1
    \end{tikzcd}
    \]
    of $\mathcal{X}$ into the projective bundle $\Pb(\mathcal{F})$. Then $(A,B)$ are characterized by the property that $\mathcal{X}$ inside $\Pb(\mathcal{F})$ is cut out by the Weierstrass equation
    \begin{equation} \label{eqn:Weierstrassequation}
    Y^2Z = X^3 + AXZ^2+BZ^3\,,
    \end{equation}
    where $Z, X, Y$ are special coordinate sections of $\mathcal{O}_{\Pb(\mathcal{F})}(1), \mathcal{L}^{\otimes 2} \otimes \mathcal{O}_{\Pb(\mathcal{F})}(1)$ and $\mathcal{L}^{\otimes 3} \otimes\mathcal{O}_{\Pb(\mathcal{F})}(1)$ defined locally over $S \times \Pb^1$ (see \cite[Section II.1]{Miranda2} for details).
    % \begin{align*}
    %     Z &\in H^0(\Pb(\mathcal{F}), \mathcal{O}_{\Pb(\mathcal{F})}(1))\,,\\
    %     X &\in H^0(\Pb(\mathcal{F}), \mathcal{L}^{\otimes 2} \otimes \mathcal{O}_{\Pb(\mathcal{F})}(1))\,,\\
    %     Y &\in H^0(\Pb(\mathcal{F}), \mathcal{L}^{\otimes 3} \otimes\mathcal{O}_{\Pb(\mathcal{F})}(1))\,
    % \end{align*}
    % are special sections.
    %where $X,Y,Z \in H^0(\Pb(\mathcal{F}), \mathcal{O}_{\Pb(\mathcal{F})}(1))$ are special sections \jocomment{to fill in details}.
    
    Now let us return to the construction of the map $\Phi$ from \eqref{eqn:Phimapdef}. Recall what we want to construct: a map $S \to \Pov$ is given by the data of a line bundle $\mathcal{N}$ on $S$ together with sections
    \begin{equation} \label{eqn:aibjsections}
    a_0, \ldots, a_{4n} \in H^0(S, \mathcal{N}^{\otimes 4}),\ b_0, \ldots, b_{6n} \in H^0(S, \mathcal{N}^{\otimes 6})
    \end{equation}
    not vanishing simultaneously, and unique up to scaling by $\lambda \in H^0(S, \mathcal{O}^{*})$ with $\lambda \cdot ((a_i), (b_j)) = ((\lambda^4 a_i), (\lambda^6 b_j))$. 
    By the Kunneth formula \cite[Tag 0BEF]{Stacks} we have that for $e=4,6$
    \begin{align*}
        H^0(S \times \Pb^1, \mathcal{L}^{\otimes e}) &= H^0(S, \mathcal{L}_0^{\otimes e}) \otimes H^0(\Pb^1, \mathcal{O}_{\Pb^1}(e))\\
        &= \bigoplus_{j=0}^{e\cdot n} H^0(S, \mathcal{L}_0^{\otimes e}) X^j Y^{en-j}\,.
    \end{align*}
    Thus from the sections $A,B$ in \eqref{eqn:ABdef} we indeed obtain unique sections $a_i \in H^0(S, \mathcal{L}_0^{\otimes 4})$, $b_j \in H^0(S, \mathcal{L}_0^{\otimes 6})$ such that
    \[
    A = \sum_{i=0}^{4n} a_i X^i Y^{4n-i},\ B = \sum_{j=0}^{6n} b_j X^j Y^{6n-j}\,.
    \]
    By the description of the data necessary for a map to $\Pov$ above, taking $\mathcal{N}=\mathcal{L}_0$, we get a morphism $S \to \Pov$. It is well-defined since both for the pair $(A,B)$ and the tuples $((a_i), (b_j))$ the scaling action had weights $4,6$ on the two factors. 

    This concludes our construction of the morphism $\Phi: \WeierP_n \to \Pov$. 
    Before moving on, let us remark for later that this map is equivariant for the natural action of $\PGL_2$ on domain and target.
    We are now ready to prove Theorem~\ref{Thm:Weierstrass_Moduli}:

    \begin{proof}[Proof of Theorem~\ref{Thm:Weierstrass_Moduli}]
    On the open substack $U_\Delta \subseteq \Pov$, the inverse $\Psi$ of the functor $\Phi$ above can be defined by the Weierstrass equation \eqref{eqn:Weierstrassequation}. Indeed, an object of $\Pov$ over a scheme $S$ is given by the data of a line bundle $\mathcal{N}$ on $S$ together with sections $a_i, b_j$ as in \eqref{eqn:aibjsections}. Forming the sections 
    \[
    A = \sum_{i=0}^{4n} a_i X^i Y^{4n-i} \in H^0(S \times \Pb^1, (\mathcal{N}(n))^{\otimes 4}),\ B = \sum_{j=0}^{6n} b_j X^j Y^{6n-j}H^0(S \times \Pb^1, (\mathcal{N}(n))^{\otimes 6})
    \]
    the map $S \to \Pov$ factors through $U_\Delta$ if the discriminant $\Delta = 4A^3 + 27B^2$ is nonzero on all fibres over $S$. Setting $\mathcal{L} = \mathcal{N}(n)$ on $S \times \Pb^1$ as well as $\mathcal{F} = \mathcal{O}_{S \times \Pb^1} \oplus \mathcal{L}^{\otimes -2} \oplus \mathcal{L}^{\otimes -3}$, we define $\mathcal{X} \subset \Pb(\mathcal{F})$ by the Weierstrass equation \eqref{eqn:Weierstrassequation}. Then the induced map $f: \mathcal{X} \to S \times \Pb^1 \to S$ together with the section $S \times \Pb^1 \to \mathcal{X}$ defined by $X=Z=0$ in the local coordinates gives the desired object of $\WeierP_n$ over $S$.
    
    The result that $\Psi$ is the inverse functor to $\Phi$ follows from the classical theory of Weierstrass fibrations as explained in \cite[Section III]{Miranda2}. In particular, we note that these functors induce isomorphisms on the level of stabilizer groups: the isomorphisms of a family $\mathcal{X} \to S \times \Pb^1$ are induced by the automorphism of $\Pb(\mathcal{F}) \supset \mathcal{X}$ given by scaling the coordinates $Z,X,Y$ by $1, \lambda^2, \lambda^3$ for $\lambda \in H^0(S,\mathcal{O}_S^\times)$. This induces a corresponding action of $\lambda$ on $A,B$ (and thus on the coefficients $a_i, b_j$) sending them to $\lambda^4 A$, $\lambda^6 B$, which are precisely the automorphisms of the data defining the map $S \to U_\Delta$.
    
    The statement that the open substack $\Umin \subseteq \UDelta$ corresponds (under $\Phi$) to the locus of minimal Weierstrass fibrations, follows from the characterization of such fibrations in terms of their defining Weierstrass equations (see \cite[Proposition III.3.2]{Miranda2}).
    
    To see that $\Usf$ gives the locus of stable fibrations, we observe that the curve cut out by the Weierstrass equation
    \[
    Y^2Z = X^3 + aXZ^2 + bZ^3
    \]
    has at worst nodal singularities if and only if one of $a,b$ is nonzero (and a cusp otherwise). Thus the condition on $\Usf$ that the sections $A,B$ do not vanish simultaneously is precisely equivalent to the statement that all fibres of the associated Weierstrass fibrations are stable genus $1$ curves. 
    %\jocomment{I think this statement is just classical, and not hard to prove, so I vote for just using the formulation I proposed.}
    % \jocomment{Similar for stable fibrations?}
    % \jucomment{WTS: The curve $V(Y^2Z = X^3 + aXZ^2 + bZ^3)$ together with the point $[0:1:0]$ is stable if and only if $(a,b) \neq 0$. Idea: This is true by the Kodaira-Neron's classification of elliptic singular fibers in relation to Tate's algorithm \cite{Tate} over $\Spec(\Zb[1/ 6])$ as we are working with short Weierstrass equation (The classification of elliptic singular fibers is unaffected for char. 2 and 3 but the Tate's algorithm is affected). Namely, the only semistable singular fibers with $g(X_t)=1$ are of the type $I_k$ as in \cite[Theorem 6.2]{Kodaira} which are denoted as the type $b_k$ in \cite[Proof of Theorem 1]{Neron}. Only semistable fibers such as $I_0$ and $I_k$ are present then by [KNT] $(a,b) \neq 0$ and vice versa.}
    % \begin{enumerate}
    %     \item $I_0$ : nonsingular elliptic (generic smooth fiber),
    %     \item $I_1$ : irreducible rational with one node (fishtail singular fiber),
    %     \item $I_{k \ge 2}$ : $k$--cycle of $(-2)$-curves ~ (necklace singular fiber).
    % \end{enumerate}
    \end{proof}

    \begin{rmk}
    \par We already discussed how points $[(A,B)] \in \Usf$ can be interpreted as morphisms $\Pb^1 \to \Pc(4,6) \iso \Me$. Then the associated Weierstrass fibration over $\Pb^1$ is obtained by pulling back the universal family $\overline{\mathcal{C}}_{1,1} \to \overline{\mathcal{M}}_{1,1}$ resulting a relatively-minimal minimal stable elliptic fibration over $\Pb^1$.

    \medskip

    Similarly, given $[(A,B)] \in \Umin \setminus \Usf$, we obtain a rational map $\Pb^1 \dashrightarrow \overline{\mathcal{M}}_{1,1}$, with indeterminacy at the common zeros of $A,B$. As a globally minimal Weierstrass model is given (which always exists over $\Pb^1$ \cite[Section 4.10.]{SS}), the type of singular fibers as classified in \cite[Theorem 6.2]{Kodaira} and \cite[Proof of Theorem 1]{Neron} at the locus of indeterminacy can be completely determined purely algebraically by Tate's algorithm \cite{Tate} (see also \cite[IV \S 9]{Silverman2}). Here the associated Weierstrass fibration is obtained by pulling back $\overline{\mathcal{C}}_{1,1} \to \overline{\mathcal{M}}_{1,1}$ away from the locus of indeterminacy and then taking a suitable completion resulting a relatively-minimal minimal unstable elliptic fibration over $\Pb^1$.

    \medskip

    Finally, for $[(A,B)] \in \UDelta \setminus \Umin$ a similar process in fact does \emph{not} result in the Weierstrass fibration given by $\Phi^{-1}([(A,B)])$, but instead in the minimal Weierstrass fibration (or \emph{normal form})  associated to  $\Phi^{-1}([(A,B)])$. The reason is that the rational map $\Pb^1 \dashrightarrow \overline{\mathcal{M}}_{1,1}$ associated to $[(A,B)]$ is insensitive to common zeros of order at least $4$ on $A$ and at least $6$ on $B$ - such common zeros can be removed by the process described in \cite[Section 4.8.]{SS} without changing the underlying rational map. And this algorithm is precisely how the defining equation of the normal form is obtained (see \cite[Lemma III.3.4]{Miranda2}).
    \end{rmk}
    
    \subsection{Fibrations over the unparameterized projective line}
    We now come to the definition of the moduli stack $\Weier$ of Weierstrass fibrations over an unparameterized $\Pb^1$ and show that it is obtained as the stack quotient of $\WeierP$ under the action of $\PGL_2$.

    \begin{defn} \label{Def:Weier}
        Given a base scheme $S$, a family of Weierstrass fibrations (over smooth rational curves) parameterized by $S$ is specified by the data
        \begin{equation} \label{eqn:WeierPfamily_unfixed}
        \mathcal{X} \xrightarrow{f} \mathcal{P} \xrightarrow{\pi} S,\ \mathcal{P} \xrightarrow{s} \mathcal{X}
        \end{equation}
        such that 
        \begin{itemize}
            \item $\pi$ is a smooth, proper morphism, locally of finite type with geometric fibres isomorphic to $\Pb^1$,
            \item $f$ is a flat, proper map with section $s$, 
            \item the fibres $(\mathcal{X}_t \xrightarrow{f} \mathcal{P}_t, \mathcal{P}_t \xrightarrow{s} \mathcal{X}_t)$ over geometric points $t$ in $S$ are Weierstrass fibrations.
        \end{itemize}
        Given a morphism $T \to S$, we have a natural pullback fibration
        \[
        \mathcal{X} \times_S T = \mathcal{X}_T \xrightarrow{f_T} \mathcal{P}_T = \mathcal{P} \times_S T \to T, \mathcal{P} \times_S T \xrightarrow{s_T} \mathcal{X}_T\,.
        \]
        We say that two fibrations $\mathcal{X}, \mathcal{X}'$ over $S$ are isomorphic if there exists isomorphisms $\varphi: \mathcal{X} \to \mathcal{X}'$, $\psi: \mathcal{P} \to \mathcal{P}'$ over $S$ making the following diagram commute:
        \begin{equation} \label{eqn:Weierdiagram}
            \begin{tikzcd}
            \mathcal{X} \arrow[d,"f"] \arrow[r,"\varphi"] & \mathcal{X}' \arrow[d,"f'"]\\
            \mathcal{P} \arrow[r,"\psi"] \arrow[u,bend left,"s"] & \mathcal{P}'\arrow[u,bend left,"s'"]
            \end{tikzcd}
        \end{equation}
        Denote by $\Weier$ the stack
        %\footnote{\jocomment{Do we want to show ourselves that this is an algebraic stack? This would require showing that families glue (stack) and that we have a cover by a scheme (algebraic stack). I think we can kind of avoid this since we show that the components of this thing are isomorphic as categories fibred in groupoids to things which are stacks.}}
        whose objects over a scheme $S$ are the families of Weierstrass fibrations over $S$ as defined above. The morphisms over $T \to S$ are given by pullback diagrams.
        
        As before, define $\Weier_n$ to be the open and closed substack of $\Weier$ parameterizing Weierstrass fibrations whose fundamental line bundle has degree $n \in \N$ which implies that a given Weierstrass fibration has the discriminant degree is $12n$.
    \end{defn}
    \begin{rmk}
    To make a connection to related notions in the literature, we point out that in \cite[Definition 7.1]{Miranda} Miranda defines the notion of a stable rational Weierstrass fibration over a scheme. Translated to our notation, it requires that the map $\mathcal{P} \to S$ should be a $\Pb^1$-bundle. However, since projective bundles on a smooth cover of a scheme together with compatible gluing maps of the underlying schemes do not necessarily glue to a projective bundle globally, the above definition is more suitable for defining a stack.
    \end{rmk}
    We are now ready to prove Theorem~\ref{Thm:Weierstrass_Stack_Quotient}:
    \begin{proof}[Proof of Theorem~\ref{Thm:Weierstrass_Stack_Quotient}]
    To prove part a), observe that there is a natural morphism $q : \WeierP_n \to \Weier_n$ (setting $\mathcal{P} = S \times \Pb^1$ on families over $S$). This morphism is invariant under the action of $\PGL_2$ on $\WeierP_n$ and thus it factors canonically through the quotient stack $p : \WeierP_n \to \WeierP_n / \PGL_2 $ as follows:
    \begin{equation} \label{eqn:torsordiag}
    \begin{tikzcd}
    \WeierP_n \arrow[d,"p"] \arrow[dr,"q"] & \\
    \WeierP_n / \PGL_2 \arrow[r,"h"] & \Weier_n
    \end{tikzcd}
    \end{equation}
    To construct the inverse map to $h$, we want to use that by \cite[Theorem 4.1]{Romagny} the quotient $\WeierP_n / \PGL_2$ is the stack of $\PGL_2$-torsors mapping equivariantly to $\WeierP_n$. Thus to construct $h^{-1}$ we only need to show that the natural $\PGL_2$-action on $\WeierP_n$ and the identity map $\WeierP_n \to \WeierP_n$ make $q$ a $\PGL_2$-torsor over $\Weier_n$.
    
    First, we check that $q$ is smooth-locally a trivial $\PGL_2$ torsor: given any family $\mathcal{X} \xrightarrow{f} \mathcal{P} \xrightarrow{\pi} S$ we can find a smooth cover $S' \to S$ such that $\mathcal{P}_{S'}$ admits three disjoint sections over $S'$ (e.g. by taking $S'$ to be the fibre product of three copies of $\mathcal{P}$ over $S$ minus the diagonal). But any smooth, proper, locally finite type morphism with geometric fibers $\Pb^1$ and three disjoint sections is actually isomorphic to a product of the base with $\Pb^1$ such that the three sections are $0,1,\infty \in \Pb^1$ (this follows from the fact that $\overline{\mathcal{M}}_{0,3} = \mathrm{Spec}(K)$). Thus the family $\mathcal{P}_{S'}$ is isomorphic to $S' \times \Pb^1$. The pullback of $q$ under the map $S' \to \Weier_n$ is given by the families of Weierstrass fibrations $\mathcal{X}' \to S' \times \Pb^1$ together with an isomorphism  to $\mathcal{X}_{S'} \to S' \times \Pb^1 = \mathcal{P}_{S'}$. After identifying $\mathcal{X}' = \mathcal{X}_{S'}$ using this data, the remaining freedom for the isomorphism is the map $\psi: S' \times \Pb^1 \to S' \times \Pb^1$ in the diagram \eqref{eqn:Weierdiagram}, which is equivalent to a map $S' \to \PGL_2$. Thus we have
    \[
    S' \times_{\Weier_n} \WeierP_n = S' \times \PGL_2,
    \]
    so that indeed the map $q$ is a locally trivial $\PGL_2$-torsor. Using this together with \cite[Theorem 4.1]{Romagny}, we can define a map $h' : \Weier_n \to \WeierP_n / \PGL_2$ in the diagram \eqref{eqn:torsordiag} and it is straightforward to check that $h,h'$ are mutually inverse. This finishes the proof of part a).
    
    We obtain the statements in part b) as a special case of Theorem \ref{Thm:Moduli_GIT}, part b). Indeed, we simply note that by Proposition \ref{Prop:stablesemistablecrit} the stacks $\Weier_{\mathrm{min},n}$ (for $n \geq 2$) and $\Weier_{\mathrm{sf},n}$ (for $n \geq 1$) are contained in the stable locus $\M_n^{(4,6),s}$.
    \end{proof}

    \begin{rmk} \label{Rmk:Weiermin1}
    In part b) of Main Theorem \ref{Thm:Weierstrass_Stack_Quotient} the stated properties of $\Weier_{\mathrm{min},n}$ do \emph{not} hold in case $n=1$, i.e. for $\Weier_{\mathrm{min},1}$. Indeed, the stack $\Weier_{\mathrm{min},1}$ contains points with positive dimensional stabilizers, thus it is no longer of Deligne--Mumford type. Using Lemma \ref{Lem:constmap} and Proposition \ref{Prop:finredstabilizercrit} one can see that these points are precisely given by the Weierstrass data $[A:B]$ in the $\PGL_2$-orbit of
    \begin{align*}
        [A:B] = [0:X Y^5], [XY^3:0], [0:X^2 Y^4] \text{ and } [a_0 X^2 Y^2 : a_1 X^3 Y^3] \text{ for }[a_0:a_1] \in \mathcal{P}(4,6)\,,
    \end{align*}
    where in each case we have a nontrivial action of $\Gb_m$ on the coordinates $X,Y$ fixing the fibrations.
    We note that these fibrations already appeared in \cite[Table 5.1]{MP} (see also \cite[Example 5.10]{Martin}) as the four types of rational extremal elliptic fibrations with two singular fibres both of which are unstable types in dual pair. One can see that the open substack $\Weier'_{\mathrm{min},1}$ of $\Weier_{\mathrm{min},1}$ obtained by removing these points is Deligne--Mumford for $\mathrm{char}(K) \nmid n$ and tame for $\mathrm{char}(K) = 0$ or $\mathrm{char}(K) > 12$. We thank Gebhard Martin for pointing out the above examples and references.
    \end{rmk}

    \medskip

    \begin{appendix}
    \section{Applications to stacks of self-maps of \texorpdfstring{$\Pb^1$}{PP\textasciicircum1}}
    \label{sec:Appen_Sil} %\texorpdfstring{$\Hom_n(\Pb^1,\Pb^1)$}{Hom n(P1,P1)}

    Using the techniques in the paper above, we can compute some motives of spaces of rational self-maps of $\mathbb{P}^1$ appearing in work \cite{Milnor, Silverman} of Milnor and Silverman.
    There, a different action of $\PGL_2$ on $\Hom_n(\Pb^1,\Pb^1)$ is considered: an element $\varphi \in \PGL_2$ acts on $f \in \Hom_n(\Pb^1,\Pb^1)$ by conjugation, sending $f$ to $\varphi \circ f \circ \varphi^{-1}$, identifying both domain and target $\Pb^1$ to be the same $\Pb^1$ by a simultaneous choice of coordinates.
    Denote by
    \begin{equation} \label{eqn:MnSilverman}
        \mathcal{M}_n = [\Hom_n(\Pb^1,\Pb^1) / \PGL_2]
    \end{equation}
    the associated quotient stack. The paper \cite{Silverman} studies the coarse moduli space $M_n$ of $\mathcal{M}_n$ and a compactification $\overline{M}_n$ of this space using the techniques of Geometric Invariant Theory \cite{GIT}, also using the group $\SL_2$ instead of $\PGL_2$ to form the quotient. The coarse moduli space $M_n = \Hom_n(\Pb^1,\Pb^1)/\! /\SL_2$ is a geometric quotient of $\Hom_n(\Pb^1,\Pb^1)$ which is affine, integral, connected, and of finite type over $\Spec~\Zb$ by \cite[Theorem 1.1.]{Silverman}. The compactification $\overline{M}_n$ is a geometric quotient for $n$ even and a categorical quotient for $n$ odd by \cite[Theorem 2.1.]{Silverman}. The moduli spaces $\Hom_n(\Pb^1,\Pb^1)/\! /\SL_2$, $\Hom_n(\Pb^1,\Pb^1)/\! /\PSL_2$ and their various compactifications have been studied in depth by many others (see \cite{DeMarco, Levy}) for analyzing the arithmetic properties of dynamical systems.

    Using the methods \& techniques in this paper, we have the following result.

    \begin{thm} \label{Thm:motivequotientSilverman}
    Let $n \geq 2$ and $\mathrm{char}(K)=0$ or $\mathrm{char}(K)>n$, then the stack $\mathcal{M}_n$ from \eqref{eqn:MnSilverman} is a smooth, irreducible, separated and tame Deligne--Mumford stack of finite type with affine diagonal over $K$.
    
    In case that the degree $n$ is an even number, we have an equality in $K_0(\mathrm{Stck}_K)$ %the Grothendieck ring of $K$-stacks
    \begin{equation} \label{eqn:Mnmotiveeqn}
        \left\{\M_{n}\right\} = \frac{\left\{\Hom_n(\Pb^1,\Pb^1)\right\}}{\left\{\PGL_2\right\}} = \Lb^{2n - 2}~,
    \end{equation}
    and thus for $\mathrm{char}(\Fb_q) > n$, the weighted point count of $\M_{n}$ over $\Fb_q$ is equal to
        $$\#_q\left(\M_{n}\right) = q^{2n - 2}~.$$
    \end{thm}
    \begin{proof}
    The fact that $\M_{n}$ is a smooth algebraic stack of finite type with affine diagonal follows from \cite[Theorem 4.1]{Romagny} as before. Similarly, the separatedness follows in the same way as described in the proof of Proposition \ref{Pro:stablesemistablestackproperties}. We are again using that $\Hom_n(\Pb^1,\Pb^1)$ lies in the stable locus of a linearized line bundle on $\Pb^{2n+1}$ by the work of \cite{Silverman}.
    
    To show that $\M_{n}$ is a tame Deligne--Mumford stack, we prove that $\PGL_2$ acts with finite, reduced and tame stabilizers at geometric points of $\Hom_n(\Pb^1,\Pb^1)$. 
    To see this note that any stabilizer of a geometric point $[f] \in \Hom_n(\Pb^1,\Pb^1)$ must fix the two closed subschemes $\mathrm{Fix}(f), \mathrm{Crit}(f) \subseteq \Pb^1$ of fixed and critical points of $f$, respectively. Here
    \[
    \mathrm{Fix}(f) = \Gamma_f \cap \Delta \subseteq \Pb^1 \times \Pb^1
    \]
    is defined as the intersection of the graph of $f$ with the diagonal $\Delta$ inside $\Pb^1 \times \Pb^1$, whereas $\mathrm{Crit}(f) = V(df)$ is the vanishing scheme of the differential $df$ of $f$. By the assumption on the characteristic, combined with the intersection theory of $\Pb^1 \times \Pb^1$, we have that $\mathrm{Fix}(f)$ is an effective Cartier divisor of degree $n+1$ in $\Pb^1$. Moreover, we claim that any point of $\Pb^1$ has multiplicity at most $n$ inside $\mathrm{Fix}(f)$. Indeed, around a point of $\Pb^1$ which we can assume to be at $0$, the scheme $\mathrm{Fix}(f)$ is cut out by the equation $f(x) = x$, so the statement follows since the degree of $f$ is $n$.
    By a similar argument, the subscheme $\mathrm{Crit}(f)$ is an effective Cartier divisor with multiplicity at most $n-1$ at each point and, by the Riemann-Hurwitz theorem, its total degree is $2n-2$. 
    
    We claim that the union $\mathrm{Fix}(f) \cup \mathrm{Crit}(f)$ contains at least $3$ points whose multiplicities are not divisible by $\mathrm{char}(K)$. Proving this, we conclude that we have tame stabilizers by Proposition \ref{Pro:stabilizerfatpoints}. But observe that by the observations about the 
    total degrees as well as bounds on the degrees at each point given above, both $\mathrm{Fix}(f)$ and $\mathrm{Crit}(f)$ contain at least two points with the desired properties. Then we simply observe that any point $q \in \mathrm{Crit}(f)$ can have multiplicity at most $1$ in $\mathrm{Fix}(f)$. Indeed, $q \in \mathrm{Crit}(f)$ implies $df(q)=0$, so that the local equation $f(x)=x$ of $\mathrm{Fix}(f)$ can have at most a simple zero at $x=q$.
    
    To see the equality in the Grothendieck ring of $K$-stacks, we want to copy the strategy from Section \ref{sec:Motive} to replace the $\PGL_2$-action by a (twisted) $\GL_2$-action on a bigger space having the same quotient. For this consider the natural action of $\GL_2$ on 
    $$U_{\Lambdavec} = H^0\left(\Pb^1, \mathcal{O}_{\Pb^1}(n)\right)^{\oplus 2}$$
    lifting the $\PGL_2$-conjugation action above. Then the diagonal torus $T \subset \GL_n$ acts with weight $n-1$. In the case that $n$ is even, we can twist this action by $\mathsf{det}^{-(n-2)/2}$ to get to the weight $1$ action satisfying $[(U_{\Lambdavec} \setminus \{0\})/T] = \Pb^{2n+1}$. Then the same proof as for Theorem~\ref{Thm:motivequotient1} gives the result.
    \end{proof}

    Our result that the motive $\{\mathcal{M}_2\} \in K_0(\mathrm{Stck}_K)$ of the fine moduli stack $\mathcal{M}_2 = [\Hom_2(\Pb^1,\Pb^1) / \PGL_2]$ is equal to $\Lb^2$ is consistent with the classical computation as in \cite[Theorem 3.1.]{Milnor} and \cite[Theorem 1.2.]{Silverman} regarding the coarse moduli space $\Hom_2(\Pb^1,\Pb^1)/\! /\SL_2 \iso \Ab^2$ of schemes over $\Spec~\Zb$. 

    \medskip

    \section{Arithmetic geometry of stabilizers of finite subschemes of \texorpdfstring{$\Pb^1$}{P1}}
    \label{Sect:stabfatpoints}

    Let $K$ be an algebraically closed field and let $Z = V(F) \subset \Pb^1$ be a zero-dimensional subscheme, defined by some nonzero homogeneous polynomial
    \begin{equation}
        F = (a_1 X + b_1 Y)^{e_1} \cdots (a_m X + b_m Y)^{e_m} \in H^0(\Pb^1, \mathcal{O}(\sum_{i=1}^m e_i))\,.
    \end{equation}
    Since $K$ is algebraically closed, the polynomial decomposes into linear factors as above and $Z$ is supported at the points $p_i = [-b_i : a_i] \in \Pb^1$, which we assume to be distinct.

    In this section, we want to study the \emph{stabilizer subgroup} $\Stab(Z)$ inside the automorphism group $\Aut(\Pb^1) = \PGL_2$. It is the largest subgroup of $\PGL_2$ satisfying that the subscheme $Z$ is invariant under this subgroup.\footnote{~For a formal definition see \cite[II.1.3.4]{GroupesAlgebriques}, where this concept is introduced as the normalizer of a subscheme under a group action. In our case, one can consider the induced action of $\PGL_2$ on the projective space $\Pb(W)$ of homogeneous polynomials up to scaling and obtain $\Stab(Z)$ as the stabilizer group of the defining polynomial $[F]$ of $Z$, see below.}
    Alternatively, one can identify the subscheme $Z$ with its defining polynomial
    \[
    [F] \in \mathbb{P}\left(H^0(\Pb^1, \mathcal{O}(\sum_{i=1}^m e_i)) \right) \eqqcolon \Pb(W)\,.
    \]
    Then for the natural action of $\PGL_2$ on $\Pb(W)$ induced by the action of $\PGL_2$ on the homogeneous coordinates $X,Y$, the stabilizer of $Z$ equals the automorphism group of the point $\mathrm{Spec}(K) \to [\Pb(W) / \PGL_2]$ associated to $F$. The study of quotients of (open subsets of) $\Pb(W)$ under $\PGL_2$ is a classical subject of invariant theory, known as the theory of \emph{binary quantics} (see e.g. \cite[Section 4.1]{GIT}).

    In the section below, we want to obtain a necessary and sufficient criterion for the stabilizer group $\mathrm{Stab}(Z)$ to be finite and reduced over $K$. In the main paper, such a criterion is necessary for proving the Deligne--Mumford property of the quotient stacks in Theorem \ref{Thm:Moduli_GIT}.\footnote{~While we do not pursue this direction further, our criterion (Proposition \ref{Pro:stabilizerfatpoints}) could also be applied directly to the study of binary quantics above: combining it with the characterization of the GIT-stable points $\Pb(W)^{s}$ in \cite[Proposition 4.1]{GIT}, it is easy to see that the stack $[\Pb(W)^s / \PGL_2]$ is of Deligne-Mumford type away from finitely many characteristics.}

    A first necessary condition for the finiteness of $\mathrm{Stab}(Z)$ is that the number $m$ of distinct points in $Z$ is at least $3$. Indeed, otherwise we can use the action of $\PGL_2$ to send $p_1$ to $0$ and $p_2$ to $\infty$ and obtain a subgroup of $\PGL_2$ conjugate to $\Stab(Z)$ which contains an entire copy of $\Gb_m$ (acting by scaling on the coordinate $Y$). 

    As we will see below, the condition $m \geq 3$ is also sufficient in characteristic zero, but complications can arise in positive characteristic. Indeed, for $\mathrm{char}(K)=p$ consider the subscheme 
    \[
    Z = V(X \cdot (X-Y)^p \cdot Y) = V(X \cdot (X^p-Y^p) \cdot Y) \subset \Pb^1\,.
    \]
    Then the stabilizer $\Stab(Z) = \mu_p \rtimes \underline{\Sym}(2)$ is non-reduced, where $\mu_p \subset \Gb_m$ acts by scaling on the coordinate $Y$. However, below we will see that by excluding finitely many characteristics (depending on the multiplicities $e_i$ of the points of $Z$), we can ensure that $\Stab(Z)$ is finite and reduced.

    \begin{lem} \label{Lem:reducedinclusion}
    There exists a natural closed embedding $\Stab(Z)_{\text{red}} \subseteq \Stab(Z_{\text{red}})$ of the underlying reduced scheme of $\Stab(Z)$ inside the stabilizer of the reduced scheme $Z_{\text{red}}$. Moreover, for $m \geq 3$ we have that $\Stab(Z_{\text{red}})$ is naturally a closed subscheme of the finite constant group scheme $\underline{\Sym}(m)$.
    \end{lem}
    \begin{proof}
    We have that $\Stab(Z)_{\text{red}}$ and $\Stab(Z_{\text{red}})$ are both closed subgroup schemes of $\PGL_2$. By the definition of the reduced subscheme $\Stab(Z)_{\text{red}}$ it thus suffices to show the above inclusion on the level of points (see \cite[Tag 0356]{Stacks}). Moreover, since the $K$-points are dense in both subschemes, it actually suffices to verify the inclusion on $K$-points. But for an element $g \in \PGL_2(K)$ fixing $[F]$, it must in particular fix the set of zeros of $F$, which is equal to $Z_{\text{red}}$.
    %\href{https://stacks.math.columbia.edu/tag/0356}

    Now assume $m \geq 3$, then by the definition of the stabilizer scheme we have a group action of $\Stab(Z_{\text{red}})$ on $Z_{\text{red}} = \{p_1, \ldots, p_m\}$ inducing a morphism
    \begin{equation} \label{eqn:symmorphism}
        \Stab(Z_{\text{red}}) \to \Aut(Z_{\text{red}})=\underline{\Sym}(m)\,.
    \end{equation}
    We want to show that the kernel of the morphism \eqref{eqn:symmorphism} is trivial. An element of the kernel must then fix each point $p_i$ separately, but for any three distinct points $p_1, p_2, p_3$ it is classical that the morphism
    \[
    \PGL_2 \to \Pb^1 \times \Pb^1 \times \Pb^1 \setminus \Delta,\ g \mapsto (g\cdot p_1, g\cdot p_2, g\cdot p_3)
    \]
    is an isomorphism. This shows that the kernel is trivial and finishes the proof.
    \end{proof}
    \begin{cor} \label{Cor:stabchar0}
    For $\mathrm{char}(K)=0$ and $m \geq 3$ we have that $\Stab(Z)$ is finite and reduced.
    \end{cor}
    \begin{proof}
    In characteristic zero, the group $\Stab(Z)$ is automatically reduced, so by Lemma \ref{Lem:reducedinclusion} we have inclusions
    \[\Stab(Z) \subseteq \Stab(Z_{\text{red}}) \subseteq \underline{\Sym}(m)\,.\]
    Since $\underline{\Sym}(m)$ is finite and reduced, so is its subgroup $\Stab(Z)$.
    \end{proof}
    We are now ready to state the full criterion.
    \begin{prop} \label{Pro:stabilizerfatpoints}
    Let $K$ be an algebraically closed field and let $Z = V(F) \subset \Pb^1$ be a zero-dimensional subscheme, defined by some nonzero homogeneous polynomial
    \begin{equation}
        F = (a_1 X + b_1 Y)^{e_1} \cdots (a_m X + b_m Y)^{e_m} \in H^0(\Pb^1, \mathcal{O}(\sum_{i=1}^m e_i))\,,
    \end{equation}
    for distinct points $p_i = [-b_i : a_i] \in \Pb^1$. Then the stabilizer group $\Stab(Z)$ is finite and reduced if and only if the number of points $p_i$ with $\mathrm{char}(K) \nmid e_i$ is at least three.
    \end{prop}
    \begin{proof}
    Having finished the case of characteristic zero in Corollary \ref{Cor:stabchar0}, we can assume $\mathrm{char}(K)=p>0$. By Lemma \ref{Lem:reducedinclusion} we have that $\Stab(Z)_{\text{red}}$ is finite and reduced. Therefore,  the same is true for $\Stab(Z)$ if and only if the tangent space to $\Stab(Z)$ at $e \in \Stab(Z)$ is trivial.

    Denote by
    \[
    F : \Stab(Z) \to \Stab(Z)^{(p)}
    \]
    the Frobenius morphism. From the fact that the derivative of $F$ at $e$ vanishes and the fact that the tangent space to the kernel $\Stab(Z)[F]$ is the kernel of the induced map of tangent spaces, it follows that
    \[
    T_e \Stab(Z) = T_e \Stab(Z)[F]\,.
    \]
    Thus we are left with showing that $\Stab(Z)[F]$ is reduced.

    For this let $R$ be an Artinian local ring over $K$ and let $g \in \PGL_2[F](R)$ be given by
    \[
    g = \begin{bmatrix}
    a & b\\ c&d
    \end{bmatrix}\text{ s.t. }F(g)=\begin{bmatrix}
    a^p & b^p\\ c^p&d^p
    \end{bmatrix} = \begin{bmatrix}
    1 & 0\\ 0&1
    \end{bmatrix} \in \PGL_2(R)\,.
    \]
    First we remark that for any $[u:v] \in \Pb^1_K$ we have that $g$ fixes the subscheme $V((vX-uY)^p) \subset \Pb^1$ and therefore any subscheme $V((vX-uY)^e)$ for $p \mid e$. Indeed we have
    \begin{align*}
        \left(v(aX+bY) - u(cX+dY) \right)^p &= v^p(a^pX^p+b^pY^p) - u^p(c^pX^p+d^pY^p) \\ &= (v^p, u^p) \cdot \begin{bmatrix}
    a^p & b^p\\ c^p&d^p
    \end{bmatrix} \cdot \begin{bmatrix}
    X^p \\ Y^p
    \end{bmatrix} \\ &\simeq  (v^p, u^p) \cdot \begin{bmatrix}
    X^p \\ Y^p
    \end{bmatrix} = (vX-uY)^p\,,
    \end{align*}
    where the symbol $\simeq$ means equality up to scaling. Therefore, all points $p_i \in Z$ with multiplicity $e_i$ divisible by $p$ are automatically left invariant under the entire group $\PGL_2[F]$. 

    Therefore assume without loss of generality that $p_1, \ldots, p_{m'}$ are the points with $e_i$ not divisible by $p$. Using the action of $\PGL_2$ we can assume that $p_1=[0:1]$, $p_2=[1:0]$ and $p_3=[1:1]$, since the stabilizer of the translation of $Z$ will be conjugate to the stabilizer of $Z$. Moreover, when considering the $R$-points of $\PGL_2[F]$ specializing to the identity over $K$, all such points must leave each individual fat point $V((u_iX+v_iY)^{e_i})$ invariant, as follows from Lemma \ref{Lem:reducedinclusion}.

    For the point $p_1 = [0:1]$ this means we require
    \[
    X^{e_i} \sim (a X + bY)^{e_i} = a^{e_i}X^{e_i} + e_i a^{e_i-1} b X^{e_i-1} Y + \ldots \,.
    \]
    From this it follows that $a$ must be invertible and, since $e_i \neq 0 \in K$ is invertible as well, we see that $b=0 \in R$. A similar computation shows that the fat point at $p_2 = [1:0]$ is invariant if and only if $c=0$. Thus in particular, for at most two points $p_i$ with $p \nmid e_i$, we have that $\Stab(Z)[F]$ contains a copy of 
    \begin{equation} \label{eqn:alphap}
    \left\{
    \begin{bmatrix}
    1 & 0\\ 0&d
    \end{bmatrix} \in \PGL_2(F)
    \right\} \cong \alpha_p
    \end{equation}
    and is thus not reduced. On the other hand, if also $p_3=[1:1]$ satisfies $p \nmid e_i$ we compute the matrices \eqref{eqn:alphap} fixing the remaining fat point. They are determined by
    \[
    (X-Y)^{e_i} = X^{e_i} - e_i X^{e_i-1}Y + \ldots  \sim (X-d Y)^{e_i} = X^{e_i} - e_i d X^{e_i-1} Y + \ldots \,.
    \]
    Comparing the coefficients of $X^{e_i}$ we see that the proportionality factor between the two sides must be $1$. Comparing the coefficients of $X^{e_i-1} Y$ shows then that $d=1$, again noting that $e_i$ is invertible. This shows that the only $R$-point of $\Stab(Z)[F]$ specializing to the identity is the identity itself, and thus $\Stab(Z)[F]$ is reduced. This finishes the proof.
    \end{proof}

    \end{appendix}

    \section*{Acknowledgements}
    The authors are indebted to Gebhard Martin for explaining to us the proofs of the results regarding the stabilizers of finite subschemes of the projective line in Appendix \ref{Sect:stabfatpoints}. We also thank Daniel Huybrechts and Burt Totaro for useful comments on a draft of this paper as well as Dori Bejleri and Changho Han for earlier helpful discussions. Jun-Yong Park was supported by the Institute for Basic Science in Korea (IBS-R003-D1) and the Max Planck Institute for Mathematics. Johannes Schmitt was supported by the SNF Early Postdoc.Mobility grant 184245 and thanks the Max Planck Institute for Mathematics in Bonn for its hospitality.
    %Jun-Yong Park was sponsored by the Institute for Basic Science in Korea (IBS-R003-D1) and thanks the Center for Geometry and Physics for its supports. 

    \vspace{+16 pt}

    \noindent Jun--Yong Park \\
    \textsc{Max-Planck-Institut f\"ur Mathematik, Vivatsgasse 7 \\ 53111 Bonn, Germany} \\
      \textit{E-mail address}: \texttt{junepark@mpim-bonn.mpg.de}

    \vspace{+16 pt}

    \noindent Johannes Schmitt \\
    \textsc{Mathematical Institute, University of Bonn, Endenicher Allee 60, 53115 Bonn, Germany} \\
      \textit{E-mail address}: \texttt{schmitt@math.uni-bonn.de}

\end{document}